\documentclass[a4paper,11pt]{amsart}
\usepackage{graphicx}
\newtheorem{lemma}{Lemma}
\newtheorem{prop}{Proposition}
\newtheorem{thm}{Theorem}
\newtheorem{cor}{Corollary}
\newtheorem*{result}{Main Theorem}
\theoremstyle{definition}
\newtheorem{defn}{Definition}
\theoremstyle{remark}
\newtheorem{rem}{Remark}

\newcounter{numl}
\newcommand{\labelnuml}{\textup{(\roman{numl})}}
\newenvironment{numlist}{\begin{list}{\labelnuml}%
{\usecounter{numl}\setlength{\leftmargin}{0pt}%
\setlength{\itemindent}{2\parindent}%
\setlength{\itemsep}{\smallskipamount}\def
\makelabel ##1{\hss \llap {\upshape ##1}}}}{\end{list}}
\newenvironment{bulletlist}{\begin{list}{\labelitemi}%
{\setlength{\leftmargin}{\parindent}\def
\makelabel ##1{\hss \llap {\upshape ##1}}}}{\end{list}}

\DeclareSymbolFont{script}{U}{eus}{m}{n}
\DeclareSymbolFontAlphabet{\mathscr}{script}
\DeclareMathSymbol{\Wedge}{0}{script}{"5E}
\DeclareMathAlphabet{\mathrmsl}{OT1}{cmr}{m}{sl}
\newcommand{\Ln}{{\mathrmsl\Lambda}}
\newcommand{\R}{{\mathbb R}}
\newcommand{\C}{{\mathbb C}}

\newcommand{\N}{{\mathbb N}}

\newcommand{\cL}{{\mathcal L}}

\newcommand{\cO}{{\mathcal O}}

\newcommand{\cS}{{\mathcal S}}
\newcommand{\tor}{{\mathfrak t}}

\newcommand{\rc}[1]{\mathit{ric}^{#1}}
\newcommand{\s}[1]{\mathit{s}_{#1}}
\newcommand{\Id}{\mathit{Id}}
\newcommand{\sym}{\mathop{\mathrm{sym}}\nolimits}
\newcommand{\Sym}{\mathrm{Sym}}
\newcommand{\trace}{\mathop{\mathrm{tr}}\nolimits}

\newcommand{\grad}{\mathop{\mathrm{grad}}\nolimits}

\renewcommand{\d}{\mathrmsl d}

\newcommand{\eps}{\varepsilon}

\newcommand{\ip}[1]{\langle #1 \rangle}
\newcommand{\ipq}[1]{\langle #1 \rangle}
\newcommand{\spn}[1]{\mathopen< #1\mathclose>}

\newcommand{\Kmap}{\boldsymbol K}
\newcommand{\ang}{\boldsymbol t}
\newcommand{\taumap}{\boldsymbol\tau}
\newcommand{\ximap}{\boldsymbol\xi}
\newcommand{\etamap}{\boldsymbol\eta}
\newcommand{\bx}{\boldsymbol x}
\newcommand{\by}{\boldsymbol y}
\newcommand{\bz}{\boldsymbol z}
\newcommand{\Proj}{\mathrm P}
\newcommand{\iI}{\sqrt{-1}}%{\boldsymbol i}
\newcommand{\xms}{u}
\begin{document}
\strut\vspace{-5mm}
\title[Ambitoric geometry I]
{Ambitoric geometry I: Einstein metrics\\
and extremal ambik\"ahler structures}
\author[V. Apostolov]{Vestislav Apostolov}
\address{Vestislav Apostolov \\ D{\'e}partement de Math{\'e}matiques\\
UQAM\\ C.P. 8888 \\ Succursale Centre-ville \\ Montr{\'e}al (Qu{\'e}bec) \\
H3C 3P8 \\ Canada}
\email{apostolov.vestislav@uqam.ca}
\author[D.M.J. Calderbank]{David M. J. Calderbank}
\address{David M. J. Calderbank \\ Department of Mathematical Sciences\\
University of Bath\\ Bath BA2 7AY\\ UK}
\email{D.M.J.Calderbank@bath.ac.uk}
\author[P. Gauduchon]{Paul Gauduchon}
\address{Paul Gauduchon \\ Centre de Math\'ematiques\\
{\'E}cole Polytechnique \\ UMR 7640 du CNRS\\ 91128 Palaiseau \\ France}
\email{pg@math.polytechnique.fr}
\date{February 2013}
\begin{abstract} We present a local classification of conformally equivalent
but oppositely oriented $4$-dimensional K\"ahler metrics which are toric with
respect to a common $2$-torus action.  In the generic case, these
``ambitoric'' structures have an intriguing local geometry depending on a
quadratic polynomial $q$ and arbitrary functions $A$ and $B$ of one variable.

We use this description to classify Einstein $4$-metrics which are hermitian
with respect to both orientations, as well a class of solutions to the
Einstein--Maxwell equations including riemannian analogues of the
Pleba\'nski--Demia\'nski metrics. Our classification can be viewed as a
riemannian analogue of a result in relativity due to R.~Debever, N.~Kamran,
and R.~McLenaghan, and is a natural extension of the classification of
selfdual Einstein hermitian $4$-manifolds, obtained independently by R.~Bryant
and the first and third authors.

These Einstein metrics are precisely the ambitoric structures with vanishing
Bach tensor, and thus have the property that the associated toric K\"ahler
metrics are extremal (in the sense of E.~Calabi).  Our main results also
classify the latter, providing new examples of explicit extremal K\"ahler
metrics. For both the Einstein--Maxwell and the extremal ambitoric structures,
$A$ and $B$ are quartic polynomials, but with different conditions on the
coefficients. In the sequel to this paper we consider global examples, and use
them to resolve the existence problem for extremal K\"ahler metrics on toric
$4$-orbifolds with $b_2=2$.
\end{abstract}
\maketitle
\vspace{-3mm}

\section*{Introduction}

Riemannian geometry in dimension four is remarkably rich, both intrinsically,
and through its interactions with general relativity and complex surface
geometry. In relativity, analytic continuations of families of lorentzian
metrics and/or their parameters yield riemannian
ones~\cite{Berard-Bergery,Page}, while concepts and techniques in one area
have analogues in the other. In complex geometry, E.~Calabi's extremal
K\"ahler metrics~\cite{Cal1} have become a focus of attention as they provide
canonical riemannian metrics on polarized complex manifolds, generalizing
constant Gauss curvature metrics on complex curves. The first nontrivial
examples are on complex surfaces.

This paper concerns a notion related both to relativity and complex surface
geometry. An \emph{ambik\"ahler structure} on a real $4$-manifold (or
orbifold) $M$ consists of a pair of K\"ahler metrics $(g_+, J_+,
\omega_+)$ and $(g_-, J_-, \omega_-)$ such that
\begin{bulletlist}
\item $g_+$ and $g_-$ induce the same conformal structure (i.e., $g_- =
f^2g_+$ for a positive function $f$ on $M$);
\item $J_+$ and $J_-$ have opposite orientations (equivalently the
volume elements $\frac1{2}\omega_+\wedge\omega_+$ and
$\frac1{2}\omega_-\wedge\omega_-$ on $M$ have opposite signs).
\end{bulletlist}
A product of two Riemann surfaces is ambik\"ahler. To obtain more interesting
examples, we suppose that both K\"ahler metrics are toric, with common torus
action, which we call ``ambitoric''. More precisely, we suppose that
\begin{bulletlist}
\item there is a $2$-dimensional subspace $\tor$ of vector fields on $M$,
linearly independent on a dense open set, whose elements are hamiltonian and
Poisson-commuting Killing vector fields with respect to both $(g_+, \omega_+)$
and $(g_-, \omega_-)$.\footnote{If $\omega$ is a symplectic form, hamiltonian
vector fields $K_1=\grad_{\omega} f_1$ and $K_2=\grad_{\omega} f_2$
\emph{Poisson-commute} iff the Poisson bracket $\{f_1,f_2\}$ with respect to
$\omega$ is zero. This holds iff $\omega(K_1,K_2)=0$.}
\end{bulletlist}

The theory of hamiltonian $2$-forms in four dimensions~\cite{ACG} implies that
any orthotoric K\"ahler metric and certain K\"ahler metrics of Calabi type are
ambitoric. Such metrics provide interesting examples of extremal K\"ahler
surfaces~\cite{Cal1,Christina1,CPTV,Hwang-Simanca,Hwang-Singer,ACG,ACGT,Eveline}.
Here we give a local classification of ambitoric structures in general, and an
explicit description of the extremal K\"ahler metrics thus unifying and
generalizing these works.

Our examples include riemannian analogues of Pleba\'nski--Demia\'nski
metrics~\cite{Pleb-Dem}; the latter are Einstein--Maxwell spacetimes of Petrov
type D, which have been extensively studied~\cite{GP}, and classified by
R.~Debever, N.~Kamran and R.~G. McLenaghan~\cite{DKM}. In riemannian geometry,
the type D condition means that both half-Weyl tensors $W^\pm$ are degenerate,
i.e., at any point of $M$ at least two of the three eigenvalues of $W^\pm$
coincide (where $W^+$ and $W^-$ are viewed as symmetric tracefree operators
acting on the three-dimensional spaces of selfdual and antiselfdual $2$-forms
respectively). Einstein metrics $g$ with degenerate half-Weyl tensors have
been classified when $W^+=0$ or $W^-=0$~\cite{AG2}---otherwise, the riemannian
Goldberg--Sachs theorem~\cite{PB,Boyer,Nurowski,AG1} and the work of
A. Derdzi\'nski~\cite{De} imply that $g$ is ambik\"ahler, with compatible
K\"ahler metrics $g_\pm = \smash{|W^\pm|_g^{2/3}}g$; conversely
$g=s_\pm^{-2}g_\pm$, where $s_\pm$ are the scalar curvatures of $g_\pm$. From
the $J_\pm$-invariance of the Ricci tensor of $g$, it follows that
$\grad_{\omega_\pm} s_\pm$ are commuting Killing vector fields for $g_\pm$,
which means that $g_{\pm}$ are both extremal K\"ahler metrics. A little more
work yields the following result.

\begin{thm}\label{thm:AHE-refined}  Let $(M,g)$ be an oriented Einstein
$4$-manifold with degenerate half-Weyl tensors $W^\pm$.  Then $g$ admits
compatible ambitoric extremal metrics $(g_\pm,J_\pm,\omega_\pm,\tor)$ near any
point in a dense open subset of $M$. Conversely, an ambik\"ahler structure is
conformally Einstein on a dense open subset if and only if its Bach tensor
vanishes.
\end{thm}

This suggests classifying such Einstein metrics within the broader context of
\emph{extremal ambik\"ahler metrics} or, equivalently, ambik\"ahler metrics
for which the Bach tensor is \emph{diagonal}, i.e., both $J_+$ and $J_-$
invariant.  We also discuss riemannian metrics of ``Pleba\'nski--Demia\'nski
type'', for which the tracefree Ricci tensor satisfies $\rc{g}_0(X,Y) = c\,
g(\omega_+(X),\omega_-(Y))$ for some constant $c$.  In particular $\rc{g}$ is
{diagonal}.  These two curvature generalizations also give rise to ambitoric
structures.

\begin{thm}\label{thm:AT-rough} An ambik\"ahler structure
$(g_\pm, J_{\pm},\omega_{\pm})$, not locally a K\"ahler product, nor of Calabi
type, nor conformal to a $\pm$-selfdual Ricci-flat metric, is locally\textup:
\begin{bulletlist}
\item ambitoric if and only if there is a compatible metric $g$ with $\rc{g}$
diagonal\textup; further, $g$ has Pleba\'nski--Demia\'nski type if and only if
it has constant scalar curvature\textup;
\item extremal and ambitoric if and only if the Bach tensor of $c$ is
  diagonal.
\end{bulletlist}
\end{thm}

Thus motivated, we study ambitoric structures in general and show that in a
neighbourhood of any point, they are either of Calabi type (hence classified
by well-known results), or ``regular''. Our explicit local classification in
the regular case (Theorem~\ref{thm:ambitoric}) relies on subtle underlying
geometry which we attempt to elucidate, although some features remain
mysterious. For practical purposes, however, the classification reduces
curvature conditions (PDEs) on ambitoric structures to systems of functional
ODEs. We explore this in greater detail in section~\ref{s:curvature}, where we
compute the Ricci forms and scalar curvatures for an arbitrary regular
ambitoric pair $(g_+, g_-)$ of K\"ahler metrics. This leads to an explicit
classification of the extremal and conformally Einstein examples
(Theorem~\ref{thm:main}). We also identify the metrics of
Pleba\'nski--Demia\'nski type among ambitoric structures
(Theorem~\ref{thm:einstein-maxwell})---their relation to Killing tensors is
discussed in Appendix~\ref{Killing}. We summarize the main results from
Theorems~\ref{thm:ambitoric}--\ref{thm:einstein-maxwell} loosely as follows.

\begin{result} Let $(g_\pm,J_\pm,\omega_\pm,\tor)$ be a regular ambitoric
structure. Then\textup:
\begin{bulletlist}
\item there is a quadratic polynomial $q$ and functions $A$ and $B$ of one
variable such that the ambitoric structure is given
by~\eqref{g0-xy}--\eqref{omegaminus-xy} \textup(and these are regular
ambitoric\textup)\textup;
\item $(g_+,J_+)$ is an extremal K\"ahler metric $\Leftrightarrow$ $(g_-,J_-)$
is an extremal K\"ahler metric $\Leftrightarrow$ $A$ and $B$ are quartic
polynomials constrained by three specific linear conditions\textup;
\item $g_\pm$ are conformally Einstein \textup(i.e., Bach-flat\textup) if and
only if they are extremal, with an additional quadratic relation on the
coefficients of $A$ and $B$\textup;
\item $g_\pm$ are conformal to a constant scalar curvature metric of
Pleba\'nski--Demia\'nski type if and only if $A$ and $B$ are quartic
polynomials constrained by three specific linear conditions \textup(different,
in general, from the extremality conditions\textup).
\end{bulletlist}
\end{result}
\begin{cor} Let $(M,g)$ be an Einstein $4$-manifold for which the half-Weyl
tensors $W^+$ and $W^-$ are everywhere degenerate. Then on a dense open subset
of $M$, the metric $g$ is locally homothetic to one of the following\textup:
\begin{bulletlist}
\item a real space form\textup;
\item a product of two Riemann surfaces with equal constant Gauss
curvatures\textup;
\item an Einstein metric of the form $s_+^{-2}g_+$, where $g_+$ is a Bach-flat
K\"ahler metric with nonvanishing scalar curvature $s_+$, described in
Proposition~\textup{\ref{p:calabi-type}} or Theorem~\textup{\ref{thm:main}}.
\end{bulletlist}
\end{cor}

In the second part of this work we shall obtain global consequences of these
local classification results. In particular, we shall resolve the existence
problem for extremal K\"ahler metrics on toric $4$-orbifolds with $b_2=2$.

\smallbreak

The first author was supported by an NSERC Discovery Grant and is grateful to
the Institute of Mathematics and Informatics of the Bulgarian Academy of
Sciences where a part of this project was realized. The second author thanks
the Leverhulme Trust and the William Gordon Seggie Brown Trust for a
fellowship when this project was conceived in 2001, and to the EPSRC for a
subsequent Advanced Research Fellowship. The authors are grateful to Liana
David and the Centro Georgi, Pisa, and to Banff International Research Station
for opportunities to meet in 2006 and 2009, when much of this work was carried
out. They thank Maciej Dunajski, Niky Kamran, Claude LeBrun and Arman
Taghavi-Chabert for very useful discussions and comments.

\section{Conformal hermitian geometry}

\subsection{Conformal hermitian structures}

Let $M$ be a $4$-dimensional manifold. A \emph{hermitian metric} on $M$ is
defined by a pair $(g,J)$ consisting of a riemannian metric $g\in
C^\infty(M,S^2T^*M)$ and an integrable almost complex structure $J\in
C^\infty(M,\mathrm{End}(TM))$, which are \emph{compatible} in the sense that
$g(J\cdot, J\cdot) = g(\cdot, \cdot)$.

The \emph{fundamental $2$-form} or \emph{K\"ahler form} $\omega^g\in
\Omega^2(M)$ of $(g,J)$ is defined by $\omega^g (\cdot, \cdot):=g(J\cdot,
\cdot)$; it is a $J$-invariant $2$-form of square-norm $2$.  The volume form
$v_g=\tfrac12\omega^g\wedge\omega^g$ induces an orientation on $M$ (the
complex orientation of $J$) for which $\omega^g$ is a section of the bundle
$\Wedge^+M$ of \emph{selfdual} $2$-forms; the bundle $\Wedge^-M$ of
\emph{antiselfdual} $2$-forms is then identified with the bundle of
$J$-invariant $2$-forms orthogonal to $\omega^g$.

For any metric $\tilde g=f^{-2}g$ conformal to $g$ (where $f$ is a positive
function on $M$), the pair $(\tilde g, J)$ is also hermitian. The \emph{Lee
form} $\theta^g\in \Omega^1(M)$ of $(g,J)$ is defined by
\begin{equation*}
\d\omega^g  = -2\theta^g \wedge \omega^g,
\end{equation*}
or equivalently $\theta^g = -\frac12 J \delta^g \omega^g,$ where $\delta^g$
is the co-differential with respect to the Levi-Civita connection $D^g$ of
$g$. Since $J$ is integrable, $\d\omega^g$ measures the deviation of $(g,J)$
from being a \emph{K\"ahler structure} (for which $J$ and $\omega^g$ are
parallel with respect to $D^g$). Thus a hermitian metric $g$ is K\"ahler iff
$\theta^g=0$. Indeed
\begin{equation}\label{DJ}
D^g_X \omega^g = J\theta^g\wedge X^\flat + \theta^g \wedge JX^\flat,
\end{equation}
where $X^\flat:=g(X,\cdot)$ denotes the $1$-form dual to the vector field $X$
(see e.g., \cite{AG1}).

If $\tilde g = f^{-2} g$, the corresponding Lee forms are linked by
$\theta^{\tilde g} = \theta^g + \d\log f$; it follows that there is a K\"ahler
metric conformal to $g$ iff $\theta^g$ is exact; locally, this is true iff
$\d\theta^g=0$, and $g$ is then uniquely determined up to homothety.

\begin{rem}
A conformally invariant (and well known) interpretation of the Lee form may be
obtained from the observation that a conformal class of riemannian metrics
determines and is determined by an oriented line subbundle of $S^2T^*M$ whose
positive sections are the riemannian metrics in the conformal class. Writing
this line subbundle as $\Ln^2:=\Ln\otimes \Ln$ (with $\Ln$ also oriented), it
is thus equivalently a bundle metric $c$ on $\Ln\otimes TM$ and the volume
form of this bundle metric identifies $\Ln^4$ with $\Wedge^4T^*M$. A metric in
the conformal class may be written $g=\ell^{-2}c$ for a positive section
$\ell$ of the line bundle $L=\Ln^*$; such an $\ell$ is called a \emph{length
scale}.

Any connection on $TM$ induces a connection on $L=(\Wedge^4TM)^{1/4}$; for
example, the Levi-Civita connection $D^g$ of $g=\ell^{-2}c$ induces the unique
connection (also denoted $D^g$) on $L$ with $D^g\ell =0$. A connection $D$ on
$TM$ is said to be \emph{conformal} if $Dc=0$.  It is well known (see
e.g.~\cite{CP}) that taking the induced connection on $L$ is an affine
bijection from the affine space of torsion-free conformal connections on $TM$
(the \emph{Weyl connections}) to the affine space of connections on $L$
(modelled on $\Omega^1(M)$).

If $J$ is hermitian with respect to $c$, the connection $D^g+\theta^g$ on $L$
is independent of the choice of metric $g=\ell^{-2}c$ in the conformal class.
Equation~\eqref{DJ} then has the interpretation that $D^J$ is the unique
torsion-free conformal connection with $D^J J=0$, while $\d\theta^g$ is the
curvature of the corresponding connection on $L$.
\end{rem}

In view of this remark, we will find it more natural in this paper to view a
hermitian structure as a pair $(c,J)$ where $c$ is a \emph{conformal metric}
as above, and $J$ is a complex structure which is orthogonal with respect to
$c$ (i.e., $c(J\cdot,J\cdot)=c(\cdot,\cdot)$). We refer to $(M,c,J)$ as a
hermitian complex surface. A \emph{compatible} hermitian metric is then given
by a metric $g=\ell^{-2}c$ in the corresponding conformal class.

\subsection{Conformal curvature in hermitian geometry}

If $(M,c)$ is an oriented conformal $4$-manifold, then the curvature of $c$,
measured by the \emph{Weyl tensor} $W\in \Omega^2(M,\mathrm{End}(TM))$,
decomposes into a sum of \emph{half-Weyl tensors} $W=W^+ + W^-$ called the
\emph{selfdual} and \emph{antiselfdual} Weyl tensors, which have the property
that for $g=\ell^{-2}c$, $(V\wedge W,X\wedge Y)\mapsto g(W^\pm_{V,W}X,Y)$ is a
section $W^\pm_g$ of $S^2_0(\Wedge^\pm M)\subset \Wedge^2T^*M\otimes
\Wedge^2T^*M$, where $S^2_0$ denotes the symmetric tracefree square. A
half-Weyl tensor $W^\pm$ is said to be \emph{degenerate} iff $W^\pm_g$ is a
pointwise multiple of $(\omega^\pm\otimes\omega^\pm)_0$ for a section
$\omega^\pm$ of $\Wedge^\pm M$, where $(\cdot)_0$ denotes the tracefree
part---equivalently, the corresponding endomorphism of $\Wedge^\pm M$ has
degenerate spectrum.

If $(M,c,J)$ is hermitian, with the complex orientation, then (with respect to
any compatible metric $g=\ell^{-2}c$) the selfdual Weyl tensor has the form
\begin{equation*}
W^+_g = \tfrac 18 \kappa^g \,(\omega^g\otimes\omega^g)_0 + J(\d\theta^g)_+\odot
\omega^g,
\end{equation*}
for a function $\kappa^g$, where $J(\d\theta^g)_+(X,Y)=(\d\theta^g)_+(JX,Y)$,
$(\d\theta^g)_+$ denotes the selfdual part, and $\odot$ denotes symmetric
product.

\begin{prop}\label{p:W-degen} If $(c,J)$ admits a compatible K\"ahler metric,
or more generally~\cite{AG1} a compatible metric $g=\ell^{-2}c$ with
$J$-invariant Ricci tensor $\rc{g}$, then $W^+$ is degenerate.
\end{prop}
This is a riemannian analogue of the Goldberg--Sachs theorem in
relativity~\cite{GS,PR}. For Einstein metrics, more information is
available~\cite{PB,Nurowski,De,AG1,LeBrun1}.
\begin{prop}\label{p:gs} For an oriented conformal $4$-manifold $(M,c)$ with a
compatible Einstein metric $g=\ell^{-2}c$, the following three conditions are
equivalent\textup:
\begin{bulletlist}
\item the half-Weyl tensor $W^+$ of $c$ is degenerate\textup;
\item every point of $(M,c)$ has a neighbourhood with a hermitian complex
structure $J$\textup;
\item every point of $(M,c)$ has a neighbourhood on which either $W^+$ is
identically zero or there is a complex structure $J$ for which $\hat
g=\smash{|W^+|^{2/3}_g} g$ is a K\"ahler metric.
\end{bulletlist}
\end{prop}
\begin{proof} The equivalence of the first two conditions is the riemannian
Goldberg--Sachs theorem~\cite{PB,Boyer,Nurowski,AG1}. Derdzi\'nski~\cite{De}
shows: if a half-Weyl tensor $W^\pm$ is degenerate, then on each connected
component of $M$ it either vanishes identically or has no zero (hence has two
distinct eigenvalues, one simple and one of multiplicity two); in the latter
case $\smash{|W^+|^{2/3}_g} g$ is a K\"ahler metric.  If $W^+$ is identically
zero on an open set $U$, there exist hermitian complex structures on a
neighbourhood of any point in $U$.
\end{proof}

\subsection{The Bach tensor}\label{s:bach}

The \emph{Bach tensor} $B$ of a $4$-dimensional conformal metric is a co-closed
tracefree section $B$ of $L^{-2}\otimes S^2T^*M$ which is the gradient of the
$L_2$-norm $\int_M |W|_c^2$ of the Weyl tensor under compactly supported
variations of the conformal metric $c$. For any compatible riemannian metric
$g=\ell^{-2}c$, $B^g=\ell^2B$ is a symmetric bilinear form on $TM$ defined by
the well-known expressions~\cite{Besse,ACG}
\begin{equation} \label{bach-general}
B^g = \delta^g \delta^g W + \tfrac12 W *_g \rc{g}_0
= 2\delta^g \delta^g W^\pm + W^\pm *_g \rc{g}_0,
\end{equation}
where we use the action of Weyl tensors $W$ (or $W^\pm$) on symmetric bilinear
forms $b$ given by $(W *_g b)(X,Y) = \sum_{i=1}^4 b(W_{X, e_i} Y, e_i)$ where
$\{e_i\}$ is a $g$-orthonormal frame.  Here $\rc{g}_0=\rc{g}-\frac14\s{g}\, g$
is the tracefree part of the Ricci tensor; the trace part does not
contribute. It immediately follows from~\eqref{bach-general} that if $W^+$ or
$W^-$ is identically zero then $c$ is \emph{Bach-flat} (i.e., $B$ is
identically zero).

The conformal invariance of $B$ implies that $B^{f^{-2}g}= f^2B^g$, while the
second Bianchi identity implies $\delta^gW=-\frac12 d^{D^g}(\rc{g}-\frac16
\s{g}\,g)$ (as $T^*M$-valued $2$-forms). Thus $c$ is also Bach-flat if it has
a compatible Einstein metric.

If $J$ is a complex structure compatible with the chosen orientation and $\hat
g$ is K\"ahler with respect to $J$, then $W^+ = \frac{1}{8}\s{\hat
g}(\omega^{\hat g} \otimes\omega^{\hat g})_0$, and the Bach tensor is easily
computed by using~\eqref{bach-general}: if $B^{\hat g,+}$ and $B^{\hat g,-}$
denote the $J$-invariant and $J$-anti-invariant parts of $B^{\hat g}$,
respectively, then (see \cite{De})
\begin{equation*}
B^{\hat g, +} = \tfrac16 (2D^+ \d\s{\hat g} + \rc{\hat g}\, \s{\hat g})_0,
\qquad
B^{\hat g, -} = - \tfrac16 D^- \d\s{\hat g},
\end{equation*}
where, for any real function $f$, $D^+ \d f$, resp. $D^- \d f$, denotes the
$J$-invariant part, resp. the $J$-anti-invariant part, of the Hessian $D^{\hat
g} \d f$ of $f$ with respect to $\hat g$, and $b_0$ denotes the tracefree part
of a bilinear form $b$. Hence the following hold~\cite{De,LeBrun1,AG1}.

\begin{prop}\label{p:bach} Let $(\hat g, J)$ be K\"ahler and let
$g= \smash{\s{\hat g}^{-2}} \, \hat g$ \textup(defined wherever $\s{\hat g}$
is nonzero\textup). Then\textup:
\begin{bulletlist}
\item $(\hat g, J)$ is \emph{extremal} \textup(i.e., $J\grad_{\hat g} \s{\hat
g}$ is a Killing vector field\textup) iff $B^{\hat g}$ is
$J$-invariant\textup;
\item $\delta^g W^+=0$ wherever $g$ is defined, and hence $B^g=W^+ *_g
\rc{g}_0$, i.e., $B^{\hat g}=\frac12 \rc{g}_0\, \s{\hat g}$\textup;
\item $g$ is an Einstein metric, wherever it is defined, iff $B^{\hat g}$ is
identically zero there.
\end{bulletlist}
\end{prop}
Thus away from zeros of $\s{\hat g}$, $\hat g$ is extremal iff $\rc{g}$ is
$J$-invariant; this generalizes.
\begin{prop}\label{p:J-inv-ric} Let $(\hat g, J)$ be K\"ahler and suppose
$g=\varphi^{-2}\hat g$ has $J$-invariant Ricci tensor. Then $J\grad_{\hat
g}\varphi$ is a Killing vector field with respect to both $g$ and $\hat g$.
\end{prop}
This follows by computing that $\rc{g}_0=\rc{\hat g}_0 + 2\varphi^{-1}(D^{\hat
g}\d\varphi)_0$.

\subsection{The Einstein--Maxwell condition}

Let $\omega_+$ and $\omega_-$ be closed (hence harmonic) selfdual and
antiselfdual $2$-forms (respectively) on an oriented riemannian $4$-manifold
$(M,g)$. Then the Einstein--Maxwell condition in general relativity has a
riemannian analogue in which the traceless Ricci tensor $\rc{g}_0$ satisfies
\begin{equation}\label{EM}
\rc{g}_0(X,Y) = c\, g(\omega_+(X), \omega_-(Y)) 
\end{equation}
for constant $c$~\cite{LeBrun3}. If $c=0$, $g$ is Einstein, while in general,
the right hand side is divergence-free, and so~\eqref{EM} implies
$\delta^g\rc{g}_0=0$, or equivalently, by the contracted Bianchi identity, $g$
is a CSC metric. A converse is available when $g$ is conformal to a K\"ahler
metric $(\hat g,J)$ with K\"ahler form $\omega^{\hat g}=\omega_+$
(cf.~\cite{LeBrun3} for the case $g=\hat g$).

\begin{prop}\label{p:einstein-maxwell} Let $(M,g,J)$ be a hermitian
$4$-manifold with $g$ conformal to a K\"ahler metric $(\hat g,\omega^{\hat
g})$.  Then $g$ satisfies the Einstein--Maxwell equation \eqref{EM}, for some
$\omega_\pm$ with $\d\omega_-=0$ and $\omega_+= \omega^{\hat g}$, iff $g$ is a
CSC metric with $J$-invariant Ricci tensor.
\end{prop}
\begin{proof} Clearly, \eqref{EM} implies that $\rc{g}$ is $J$-invariant.
Writing $\hat g = f^{-2} g$, \eqref{EM} with $\omega_+=\omega^{\hat g}$, is
then equivalent to $\omega_+(f^4\,\mathrm{Ric}^g_0(\cdot),\cdot)$ being a
constant multiple of $\omega_-$, where
$\rc{g}_0(X,Y)=g(\mathrm{Ric}^g_0(X),Y)$.  Thus we require that $\omega^{\hat
g}(f^4\,\mathrm{Ric}^g_0(\cdot),\cdot)$ is closed, or equivalently co-closed.
However, the conformal invariance of the divergence on symmetric traceless
tensors of weight $-4$ implies that $\delta^{\hat g}(f^4 \,\mathrm{Ric}^g_0) =
f^6\,\delta^g\mathrm{Ric}^g_0$.  Hence (since $\omega^{\hat g}$ is $D^{\hat
g}$-parallel)~\eqref{EM} holds iff $\rc{g}_0$ is $J$-invariant and
divergence-free.
\end{proof}

\section{Ambik\"ahler $4$-manifolds and Einstein metrics}

\subsection{Ambihermitian and ambik\"ahler structures}

\begin{defn} Let $M$ be a $4$-manifold. An \emph{ambihermitian structure} is
a triple $(c,J_+,J_-)$ consisting of a conformal metric $c$ and two
$c$-orthogonal complex structures $J_\pm$ such that $J_+$ and $J_-$ induce
opposite orientations on $M$.\footnote{The prefix \emph{ambi-} means ``on both
sides'', often left and right: ambihermitian structures have complex
structures of either handedness (orientation); they should be contrasted (and
not confused) with \emph{bihermitian structures} where $J_\pm$ induce the
\emph{same} orientation on $M$.}

A compatible metric $g=\ell^{-2}c$ is called an \emph{ambihermitian metric} on
$(M,J_+,J_-)$ and we denote by $\omega_\pm^g$ (resp.~$\theta_\pm^g$) the
fundamental $2$-forms (resp.~the Lee forms) of the hermitian metrics
$(g,J_\pm)$. A symmetric tensor $S\in C^\infty(M,S^2T^*M)$ is \emph{diagonal}
if it is both $J_+$ and $J_-$ invariant.
\end{defn}
The following elementary and well-known observation will be used throughout.
\begin{lemma}\label{doubly-almost-complex}
Let $M$ be a $4$-manifold endowed with a pair $(J_+,J_-)$ of almost complex
structures inducing different orientations on $M$. Then $M$ admits a conformal
metric $c$ for which both $J_+$ and $J_-$ are orthogonal iff $J_+$ and $J_-$
commute.  In this case, the tangent bundle $TM$ splits as a $c$-orthogonal
direct sum
\begin{equation*}
TM = T_+M \oplus T_-M
\end{equation*}
of $J_\pm$-invariant rank $2$ subbundles $T_\pm M$ defined as the $\pm
1$-eigenbundles of $-J_+J_-$.  \textup(Thus a tangent vector $X$ belongs to
$T_{\pm}M$ iff $J_+X=\pm J_- X$.\textup)
\end{lemma}

It follows that an ambihermitian metric $g$ is equivalently given by a pair of
commuting complex structures on $M$ and hermitian metrics on each of the
complex line subbundles $T_+M$ and $T_-M$. Also any diagonal symmetric tensor
$S$ may be written $S(X,Y)=f\, g(X,Y)+h\, g(J_+J_-X,Y)$ for functions $f,h$.

\begin{defn}\label{d:ambikahler}
An ambihermitian conformal $4$-manifold $(M,c,J_+, J_-)$ is called
\emph{ambik\"ahler} if it admits ambihermitian metrics $g_+$ and $g_-$ such
that $(g_+,J_+)$ and $(g_-, J_-)$ are K\"ahler metrics.
\end{defn}
With slight abuse of notation, we denote henceforth by $\omega_+$ and
$\omega_-$ the corresponding (symplectic) K\"ahler forms, thus omitting the
upper indices indicating the corresponding K\"ahler metrics $g_+$ and
$g_-$. Similarly we set $v_{\pm}=\tfrac12\omega_{\pm}\wedge\omega_{\pm}$.

\subsection{Type $D$ Einstein metrics and Bach-flat ambik\"ahler structures}
\label{sbachflat}

Proposition~\ref{p:W-degen} shows that ambik\"ahler structures have degenerate
half-Weyl tensors. A converse is available for $4$-dimensional Einstein
metrics with degenerate half-Weyl tensors $W^\pm$ (riemannian analogues of
Petrov type D vacuum spacetimes).

If $W^\pm$ both vanish, then $g$ has constant curvature, i.e., is locally
isometric to $S^4, {\R}^4$ or $H^4$, hence locally ambik\"ahler.  If instead
$g$ is half conformally-flat but not conformally-flat, we can assume
(reversing orientation if necessary) that $W^-=0$, $W^+\neq 0$.  Then, $W^+$
is degenerate iff $g$ is an selfdual Einstein hermitian metric (see \cite{AG2}
for a classification). In either case, the underlying conformal structure of
the Einstein metric is ambik\"ahler with respect to some hermitian structures
$J_\pm$ (see also the proof of Theorem~\ref{thm:AHE-refined} below). In the
case that $W^+$ and $W^-$ are both nonvanishing and degenerate, we may apply
Proposition~\ref{p:gs} to obtain a canonically defined ambik\"ahler
structure. The following proposition summarizes the situation.

\begin{prop}\label{p:AHE-rough}
For an oriented conformal $4$-manifold $(M,c)$ with a compatible Einstein
metric $g=\ell^{-2}c$, the following three conditions are equivalent\textup:
\begin{bulletlist}
\item both half-Weyl tensors $W^+$ and $W^-$ are degenerate\textup;
\item about each point of $M$ there exists a pair of complex
structures $J_+$ and $J_-$ such that $(c,J_+,J_-)$ is ambihermitian\textup;
\item about each point $M$ there exists a pair of complex
structures $J_+$ and $J_-$ such that $(c,J_+,J_-)$ is ambik\"ahler.
\end{bulletlist}
If $M$ is simply connected and $W^\pm$ are both nonzero, then the
compatible ambik\"ahler structure $(J_+,J_-)$ is unique \textup(up to signs
of $J_\pm$\textup) and globally defined.
\end{prop}

We now characterize Einstein metrics among ambik\"ahler structures.

\begin{prop}\label{p:AK-bach-flat} Let $(M, c, J_+,J_-)$ be a connected
ambik\"ahler $4$-manifold. Then $c$ is Bach-flat iff there is a compatible
Einstein metric $g=\ell^{-2}c$ defined on a dense open subset of $M$.
\end{prop}
\begin{proof} If $c$ is Bach-flat then by Proposition~\ref{p:bach} both of the
K\"ahler metrics $(g_+,J_+)$ and $(g_-,J_-)$ are extremal, so their scalar
curvatures, $s_+$ and $s_-$ have holomorphic gradients. By the unique
continuation principle, each of $s_\pm$ is either nonvanishing on an open
dense subset of $M$ or is identically zero. Hence if neither of the conformal
Einstein metrics $s_+^{-2}g_+$ and $s_-^{-2}g_-$ are defined on a dense open
subset of $M$, $s_\pm$ are both identically zero, which implies $W^\pm=0$;
then $c$ is a flat conformal structure and there are compatible Einstein
metrics on any simply connected open subset of $M$.

Conversely if there is compatible Einstein metric on a dense open subset,
then, as already noted, $B$ vanishes identically there, hence everywhere by
continuity.
\end{proof}

The following lemma provides a practical way to apply this characterization.

\begin{lemma}\label{l:bach-flat} Let $(M,c,J_+,J_-)$ be a connected
ambik\"ahler conformal $4$-manifold which is not conformally-flat and for
which the corresponding K\"ahler metrics $g_+$, $g_-$ are extremal, but not
homothetic. Then $c$ is Bach-flat iff the scalar curvatures $s_\pm$ of $g_\pm$
are related by
\begin{equation}\label{bflat}
C_+ s_- = C_- \Bigl(\frac{-v_-}{v_+}\Bigr)^{1/4}s_+,
\end{equation}  
where $C_\pm$ are constants not both zero and $v_\pm$ are the volume forms of
$(g_\pm,J_\pm)$.
\end{lemma}
\begin{proof} If $s_+$ or $s_-$ is identically zero, $(M,c)$ is
half-conformally-flat and (with $C_+$ or $C_-$ zero) the result is
trivial. Otherwise, if $c$ is Bach-flat, $s_+^{-2} g_+$ and $s_-^{-2} g_-$ are
Einstein metrics defined on open sets with dense intersection, so they must be
homothetic, since $(M,c)$ is not conformally-flat. Thus
\begin{equation} \label{homothetic}
s_+^{-2}  g_+ = C \, s_-^{- 2} g_-,
\end{equation}
for a positive real number $C$, and~\eqref{bflat} holds with $(C_-/C_+)^2=C$.

Conversely, with $s_\pm\neq 0$,~\eqref{bflat} implies~\eqref{homothetic}, and
we may choose $g_\pm$ so that $g:=s_+^{-2} g_+ = s_-^{-2} g_-$ (i.e.,
$C=1$). By Proposition~\ref{p:bach}, $\delta^g W^+=0=\delta^g W^-$ and
\begin{equation} \label{BS}
B^g = W^{\pm} *_g \rc{g}_0.
\end{equation}
Moreover, since $(g_+, J_+)$ and $(g_-, J_-)$ are both extremal by assumption,
$\rc{g}_0$ is diagonal, hence a pointwise multiple $\kappa$ of $(J_+J_-)_g:=
g(J_+J_- \cdot,\cdot)$.  Relation (\ref{BS}) can then be rewritten as
\begin{equation*}
B^g = \kappa\, W^\pm *_g (J_+ J_-)_g
= \kappa\, W^\pm *_{g_\pm} (J_+ J_-)_{g_\pm}
= \tfrac16 \kappa\, s_\pm (J_+ J_-)_{g_\pm}
= \tfrac16\kappa\, s_\pm^3(J_+ J_-)_g.
\end{equation*}
We deduce that $\kappa \, s_+^3 = \kappa \, s_-^3$. Since $g_+$ and $g_-$ are
not homothetic, $s_+$ and $s_-$ are not identical; since they have holomorphic
gradients, they are then not equal on a dense open set. Thus $\kappa=0$, $g$
is Einstein and $B^g=0$.
\end{proof}

\section{Ambitoric geometry}

Ambitoric geometry concerns ambik\"ahler structures for which both K\"ahler
metrics are toric with respect to a common $T^2$-action; the pointwise
geometry is the following.

\begin{defn}\label{d:ambitoric}
An ambik\"ahler $4$-manifold $(M,c,J_+, J_-)$ is said to be \emph{ambitoric}
iff it is equipped with a $2$-dimensional family $\tor$ of vector fields
which are linearly independent on a dense open set, and are Poisson-commuting
hamiltonian Killing vector fields with respect to both K\"ahler structures
$(g_\pm,J_\pm,\omega_{\pm})$.
\end{defn}
Hamiltonian vector fields $K=\grad_{\omega} f$ and $\tilde K=\grad_\omega
\tilde f$ Poisson commute (i.e., $\{f,\tilde f\}=0$) iff they are
\emph{isotropic} in the sense that $\omega(K,\tilde K)=0$; it then follows
that $K$ and $\tilde K$ commute (i.e., $[K,\tilde K]=0$). Thus $\tor$ is an
abelian Lie algebra under Lie bracket of vector fields.

We further motivate the definition by examples in the following subsections.

\subsection{Orthotoric K\"ahler surfaces are ambitoric}\label{s:orthotoric}

\begin{defn}\label{d:orthotoric}\cite{ACG}
A K\"ahler surface $(M,g,J)$ is \emph{orthotoric} if it admits two independent
hamiltonian Killing vector fields, $K_1$ and $K_2$, with Poisson-commuting
momenta $x+y$ and $xy$, respectively, where $x$ and $y$ are smooth functions
with $\d x$ and $\d y$ orthogonal.
\end{defn}
The following result is an immediate corollary to~\cite[Props.~8 \&
9]{ACG}.

\begin{prop}\label{p:orthotoric} Any orthotoric K\"ahler surface
$(M,g_+, J_+, K_1, K_2)$ admits a canonical opposite hermitian structure $J_-$
\textup(up to sign\textup) with respect to which $M$ is ambitoric with
$\tor=\spn{\{K_1,K_2\}}$.
\end{prop}

\subsection{Ambitoric K\"ahler surfaces of Calabi type}\label{s:calabi-type}

\begin{defn}\label{d:calabi-type}\cite{ACG}
A K\"ahler surface $(M,g_+,J_+)$ is said to be of \emph{Calabi type} if it
admits a nonvanishing hamiltonian Killing vector field $K$ such that the
negative almost-hermitian pair $(g_+,J_-)$---with $J_-$ equal to $J_+$ on the
distribution spanned by $K$ and $J_+K$, but $-J_+$ on the orthogonal
distribution---is conformally K\"ahler.
\end{defn}
Thus, any K\"ahler surface of Calabi type is canonically ambik\"ahler. An
explicit formula for K\"ahler metrics of Calabi type, using the LeBrun normal
form~\cite{LeBrun0} for a K\"ahler metric with a hamiltonian Killing vector
field, is obtained in \cite[Prop.~13]{ACG}: $(g_+, J_+, \omega_+)$ is given
locally by
\begin{equation}\label{calabi-type}
\begin{split}
g_+ &= (az-b)g_{\Sigma} + w(z)\d z^2 + w(z)^{-1}(\d t + \alpha)^2, \\
\omega_+ &= (az-b)\omega_{\Sigma} +  \d z\wedge (\d t + \alpha),  \quad
\d\alpha = a \omega_{\Sigma},
\end{split}
\end{equation}
where $z$ is the momentum of the Killing vector field, $t$ is a function on
$M$ with $\d t(K)=1$, $w(z)$ is function of one variable, $g_{\Sigma}$ is a
metric on a $2$-manifold $\Sigma$ with area form $\omega_{\Sigma}$, $\alpha$
is a $1$-form on $\Sigma$ and $a,b$ are constant.

The second conformal K\"ahler structure is then given by
\begin{equation*}
\begin{split}
g_-&=(az-b)^{-2}g_+,\\
\omega_- &= (az-b)^{-1} \omega_{\Sigma} - (az-b)^{-2} \d z\wedge (\d t+\alpha).
\end{split}
\end{equation*}

Note that the $\big(\Sigma, (az-b)\omega_{\Sigma}, (az-b)g_{\Sigma}\big)$ is
identified with the \emph{K\"ahler quotient} of $(M,g_+,\omega_+)$ at the value
$z$ of the momentum. We conclude as follows.

\begin{prop}\label{p:calabi} An ambik\"ahler structure of Calabi type is
ambitoric---with respect to Killing vector fields $K_1,K_2$ with $K\in
\spn{\{K_1,K_2\}}$---iff $(\Sigma, g_{\Sigma}, \omega_{\Sigma})$ admits a
hamiltonian Killing vector field.
\end{prop}
We shall refer to ambitoric $4$-manifolds arising locally from
Proposition~\ref{p:calabi} as \emph{ambitoric K\"ahler surfaces of Calabi
type}.  A more precise description is as follows.
\begin{defn}\label{d:ambitoric-calabi-type} An ambitoric $4$-manifold
$(M,c, J_+,J_-)$ is said to be of Calabi type if the corresponding
$2$-dimensional family of vector fields contains one, say $K$, with respect to
which the K\"ahler metric $(g_+, J_+)$ (equivalently, $(g_-,J_-)$) is of
Calabi type on the dense open set where $K$ is nonvanishing; without loss, we
can then assume that $J_+= J_-$ on $\spn{\{ K,J_+K \}} $.
\end{defn}
Note that this definition includes the case of a local K\"ahler product of two
Riemann surfaces each admitting a nontrivial Killing vector field (when we
have $a=0$ in \eqref{calabi-type}). In the non-product case we can assume
without loss $a=1, b=0$; hence
\begin{equation}\label{calabi-type-non-product}
\begin{split}
g_+ &= zg_{\Sigma} + \frac{z}{V(z)}\d z^2 + \frac{V(z)}{z}(\d t + \alpha)^2,\\
\omega_+ &= z\omega_{\Sigma} + \d z\wedge (\d t + \alpha), 
\qquad \d\alpha = \omega_{\Sigma},
\end{split}
\end{equation}
while the other K\"ahler metric $(g_-= z^{-2}g_+, J_-)$ is also of Calabi type
with respect to $K=\partial/\partial t$, with momentum ${\bar z}= z^{-1}$ and
${\bar V}(\bar z)= {\bar z}^4V(1/{\bar z})=V(z)/z^4$.

The form \eqref{calabi-type-non-product} of a non-product K\"ahler metric of
Calabi type is well adapted to curvature computations. For this paper, we need
the following local result.

\begin{prop}\label{p:calabi-type} Let $(M,g_+, J_+)$ be a non-product
K\"ahler surface of Calabi type with respect to $K$. Denote by $J_-$ the
corresponding negative hermitian structure and by $g_-$ the conformal K\"ahler
metric with respect to $J_-$.
\begin{bulletlist}
\item $(g_+,J_+)$ is extremal iff $(g_-,J_-)$ is extremal and this happens
precisely when $(\Sigma, g_{\Sigma})$ in \eqref{calabi-type-non-product} is of
constant Gauss curvature $k$ and $V(z)=a_0z^4 + a_1z^3 + kz^2 + a_3z +a_4$. In
particular, $(c,J_+,J_-)$ is locally ambitoric.
\item The conformal structure is Bach-flat iff, in addition,
$4a_0a_4 - a_1a_3=0$.
\item $(g_+,J_+)$ is CSC iff it is extremal with $a_0=0$, and
  K\"ahler--Einstein iff also $a_3=0$.
\end{bulletlist}
\end{prop}
\begin{proof} The result is well-known under the extra assumption that the
scalar curvature $s_+$ of the extremal K\"ahler metric $g_+$ is a Killing
potential for a multiple of $K$ (see e.g., \cite[Prop.~14]{ACG}). However, one
can show~\cite[Prop.~5]{ACGT-survey} that the later assumption is, in fact,
necessary for $g_+$ to be extremal.
\end{proof}

\subsection{Ambik\"ahler metrics with diagonal Ricci tensor}

If an ambihermitian metric $(g,J_+,J_-)$ has diagonal Ricci tensor $\rc{g}$,
then by Proposition~\ref{p:W-degen}, $W^\pm$ are degenerate, and hence the Lee
forms $\theta^g_\pm$ of $J_\pm$ have the property that $\d\theta^g_+$ is
antiselfdual, while $\d\theta^g_-$ is selfdual. Let us suppose that
$\d\theta^g_\pm=0$, so that $(g,J_+,J_-)$ is locally ambik\"ahler.  (We have
seen that this holds if $g$ is Einstein, but it is also automatic if $M$ is
compact, or if $\theta^g_+ + \theta^g_-$ is closed.)

On an open set where the K\"ahler metrics $g_\pm=\varphi_{\pm}^2g$---with
K\"ahler forms $\omega_\pm=g_\pm(J_\pm\cdot,\cdot)$---are defined,
Proposition~\ref{p:J-inv-ric} implies that $\varphi_\pm$ are Killing
potentials with respect to $(g_\pm,J_\pm)$ respectively. The corresponding
hamiltonian Killing vector fields $Z_\pm=\grad_{\omega_\pm} \varphi_\pm$ are
also Killing vector fields of $g$, since they preserve $\varphi_{\pm}$
respectively. Hence they also preserve $\rc{g}$, $W^+$ and $W^-$. We shall
further suppose that $Z_+$ preserves $J_-$, which is automatic unless $g$ is
selfdual Einstein, and that $Z_-$ preserves $J_+$, which is similarly
automatic unless $g$ is antiselfdual Einstein.

\begin{prop}\label{p:diagonal-ambi}
Let $(g_\pm,J_\pm,\omega_\pm)$ be ambik\"ahler, and suppose
$g=\varphi_\pm^{-2}g_\pm$ is a compatible metric with diagonal Ricci
tensor such that $Z_\pm=\grad_{\omega_\pm}\varphi_\pm$ preserve both $J_+$ and
$J_-$. Then precisely one of the following cases occurs\textup:
\begin{numlist}
\item $Z_+$ and $Z_-$ are both identically zero and then $(M,c,J_+,J_-)$ is a
locally a K\"ahler product of Riemann surfaces\textup;
\item $Z_+\otimes Z_-$ is identically zero, but $Z_+$ and $Z_-$ are not both
identically zero, and then $(M,c,J_+,J_-)$ is either orthotoric or of Calabi
type\textup;
\item $Z_+\wedge Z_-$ is identically zero, but $Z_+\otimes Z_-$ is not, and
then $(M,c,J_+,J_-)$ is either ambitoric or of Calabi type\textup;
\item $Z_+\wedge Z_-$ is not identically zero, and then $(M,c,J_+,J_-)$ is
ambitoric.
\end{numlist}
In particular $(M,c,J_+,J_-)$ is either a local product, of Calabi type, or
ambitoric.
\end{prop}
\begin{proof} We first note that $Z_+$ and $Z_-$ preserve both Lee forms
$\theta_\pm^g=\varphi_\pm^{-1}\d\varphi_\pm$, and hence
$\theta_\pm^g(Z_\mp)Z_\pm+[Z_\mp,Z_\pm]=0$, with $\theta_\pm^g(Z_\mp)=c_\pm$
constant. Hence $c_+ Z_+ + c_- Z_- = 0$, so $[Z_+,Z_-]=0$ and $c_\pm Z_\pm=0$,
which forces $c_\pm=0$ (since $Z_\pm=0$ implies $\theta^g_\pm=0$).  We now
have $\d\varphi_\pm(Z_\mp)=0$, so $\omega_\pm(Z_+,Z_-)=0$.

By connectedness and unique continuation for holomorphic vector fields,
conditions (i)--(iv) are mutually exclusive and the open condition in each
case holds on a dense open set. Case (i) is trivial: here $g=g_+=g_-$ is
K\"ahler and $J_+J_-$ is a $D^g$-parallel product structure.

In case (ii) either $Z_+$ or $Z_-$ is zero on each component of the dense open
set where they are not both zero. Suppose, without loss that $Z_+=0$ so that
$g=g_+$ and $Z_-=J_-\grad_{g_-}\varphi_-=J_-\grad_g\lambda$ with
$\lambda=-1/\varphi_-$. However, since $Z_-$ also preserves $\omega_+$,
$J_+J_-\d\lambda$ is closed, hence locally equal to $\frac12 \d\sigma$ for a
smooth function $\sigma$. According to \cite[Remark~2]{ACG}, the $2$-form
$\varphi := \frac{3}{2}\sigma \omega_+ + \lambda^3 \omega_-$ is hamiltonian
with respect to the K\"ahler metric $(g_+, J_+)$; by~\cite[Theorems~1 \&
3]{ACG}, this means that $g= g_+$ is either orthotoric (on a dense open
subset of $M$), or is of Calabi type.

In case (iii) $Z_+$ and $Z_-$ are linearly dependent, but are both
nonvanishing on a dense open set. Hence, we may assume, up to rescaling on
each component of this dense open set, that $Z:=Z_+=Z_-$. This is equivalent
to
\begin{equation}\label{s-relation}
J_+ \Bigl(\frac{\d\varphi_+}{\varphi_+^2}\Bigr) = J_-
\Bigr(\frac{\d\varphi_-}{\varphi_-^2}\Bigr),
\end{equation}
and hence also
\begin{equation*}
2J_\pm \d\Bigl(\frac{1}{\varphi_+\varphi_-}\Bigr)=J_{\mp}
\d\Bigl(\frac{1}{\varphi_+^2} +\frac{1}{\varphi_-^2}\Bigr).
\end{equation*}
Since $h\,g$, with $h=1/\varphi_+\varphi_-$, is the barycentre of $g_+$ and
$g_-$, it follows (cf.~\cite{Jelonek3b} and Appendix~\ref{A:ckJ}) that the
symmetric tensor $g(S\cdot,\cdot)$, where $S= f \Id + h J_+J_-$ and
$2f=1/\varphi_+^2 + 1/\varphi_-^2$, is a Killing tensor with respect to
$g$. Clearly $\cL_Z S = 0$, and it follows from~\eqref{s-relation} that $D^g
Z^\flat$ is both $J_+$ and $J_-$ invariant. Thus $X\mapsto D^g_X Z$ commutes
with $S$ and $D^g_ZS=0$. Straightforward computations now show that $SZ$ is a
Killing field with respect to $g$, and hamiltonian with respect to
$\omega_\pm$.

Moreover, $Z$ and $SZ$ commute and span a isotropic subspace with respect to
$\omega_\pm$, so define an ambitoric structure on the open set where they are
linearly independent.  Clearly $Z$ and $SZ$ are linearly dependent only where
$J_+J_-Z$ is proportional to $Z$, in which case $g_\pm$ is of Calabi type.

Case (iv) follows by definition.
\end{proof}

\smallbreak
\noindent {\it Proof of Theorem}~\ref{thm:AT-rough}. For the first part, if
$g= \varphi_{\pm}^{-2} g_{\pm}$ has diagonal Ricci tensor,
Proposition~\ref{p:diagonal-ambi} implies the existence of an ambitoric
structure once we show that $\cL_{Z_+} J_- = 0 = \cL_{Z_-}J_+$ where
$Z_\pm=\grad_{\omega_\pm} \varphi_\pm= -J_{\pm} \grad_{g}\varphi_{\pm}^{-1}$
are the corresponding Killing vector fields of $g$. As already observed, this
is automatic unless $g$ is Einstein and (anti)selfdual. By assumption and
without loss of generality, we may suppose $g$ is a selfdual Einstein metric
with nonzero scalar curvature $\s g$ which is not antiselfdual. As $W^+$ does
not vanish identically, it determines $J_+$ up to sign, and so $\cL_{Z_-}
J_+=0$. Since $Z_+ = -J_+ \grad_g |W^+|^{-1/3}_g$ it follows that $[Z_-,
Z_+]=0$.  In order to show $\cL_{Z_+} J_-=0$, we recall that negative K\"ahler
metrics $g_-$ in the conformal class are in a bijection with antiselfdual
twistor $2$-forms $\psi$ (see~\cite{Pontecorvo} and Appendix~\ref{Killing}),
the latter being defined by the property that there is a $1$-form $\alpha$
such that $D^g_X \psi = (\alpha \wedge X^{\flat})^-$ for any vector field $X$,
where $(\cdot )^-$ denotes the antiselfdual part. Specifically, in our case,
$\psi= \varphi_-^{-1} \omega_-$ and $\alpha= 2Z_-^{\flat}$. Since $\cL_{Z_+}
Z_-^{\flat} =0$, $\cL_{Z_+} \psi$ is a parallel antiselfdual $2$-form. As $g$
is selfdual with nonzero scalar curvature, the Bochner formula shows there are
no non-trivial parallel antiselfdual $2$-forms; hence $\cL_{Z_+} \psi=0$ and
so $\cL_{Z_+} J_-=0$.

In the other direction, we shall see later in
Proposition~\ref{p:diagonal-ricci} that any regular ambitoric structure admits
compatible metrics with diagonal Ricci tensor. The characterization of the
Pleba\'nski-Demia\'nski case now follows from
Proposition~\ref{p:einstein-maxwell}.

For the second part, Proposition~\ref{p:bach} implies that an ambik\"ahler
structure $(g_{\pm},\omega_{\pm},J_{\pm})$ has diagonal Bach tensor iff both
K\"ahler metrics are extremal. The assumption on the conformal class ensures
that it is not conformally flat and hence the corresponding scalar curvatures
$s_{\pm}$ do not both vanish identically, so that, using
Proposition~\ref{p:bach} again, the metric $g=s_+^{-2}g_+$ say is well-defined
with diagonal Ricci tensor on a dense open subset of $M$.  By
Proposition~\ref{p:diagonal-ambi} (noting that $Z_{\pm} =J_{\pm}
\grad_{g_{\pm}} s_{\pm}$ are well-defined on $M$) we conclude that
$(g_{\pm},\omega_{\pm},J_{\pm})$ is ambitoric. \qed

\subsection{Ambihermitian Einstein $4$-manifolds are locally ambitoric}
\label{s:Eambihermitian}

Proposition~\ref{p:AHE-rough} implies that any Einstein metric with degenerate
half Weyl tensors---in particular, any ambihermitian Einstein metric---is
ambik\"ahler and Bach-flat. Conversely, Bach-flat ambik\"ahler metrics
$(g_\pm,J_\pm)$ are conformal to an Einstein metric $g$ on a dense open set by
Proposition~\ref{p:AK-bach-flat}

In the generic case that $W^\pm$ are both nonzero, the ambik\"ahler metrics
conformal to $g$ are $g_\pm=\smash{|W^\pm|_g^{2/3}} g$, and the Einstein
metric is recovered up to homothety as $g=s_{\pm}^{-2}g_\pm$, where $s_\pm$ is
the scalar curvature of $g_\pm$. We have already noted that the vector fields
$Z_\pm:= J_\pm \grad_{g_{\pm}}s_\pm$ are Killing with respect to $g_\pm$
(respectively) and hence also $g$. More is true.

\begin{prop}\label{p:generic-ein} Let $(M,c,J_+,J_-)$ be a Bach-flat
ambik\"ahler manifold such that the K\"ahler metrics $g_\pm$ have nonvanishing
scalar curvatures $s_\pm$. Then the vector fields $Z_\pm=J_\pm
\grad_{g_{\pm}}s_\pm$ are each Killing with respect to both $g_+$ and $g_-$,
holomorphic with respect to both $J_+$ and $J_-$, and isotropic with respect
to both $\omega_+$ and $\omega_-$ \textup(i.e., $\omega_\pm(Z_+, Z_-)=
0$\textup{);} in particular $Z_+$ and $Z_-$ commute.

Furthermore $(M,c,J_+,J_-)$ is ambitoric in a neighbourhood of any point in a
dense open subset, and on a neighbourhood of any point where $Z_+$ and $Z_-$
are linearly independent, we may take $\tor=\spn{\{Z_+, Z_-\}}$.
\end{prop}
\begin{proof}
$Z_+$ and $Z_-$ are conformal vector fields, so they preserve $W^\pm$ and its
unique simple eigenspaces. One readily concludes~\cite{AG1,De} that the Lie
derivatives of $g_+$, $g_-$, $J_+$, $J_-$ (and hence also $\omega_+$ and
$\omega_-$) all vanish. Consequently, $\cL_{Z_+}s_-=0=\cL_{Z_-}s_+$---or
equivalently $\omega_\pm(Z_+, Z_-)= 0$. This proves the first part.

Since we are now in the situation of Proposition~\ref{p:diagonal-ambi}, it
remains to show that $(M,c,J_+,J_-)$ is locally ambitoric even in cases where
Proposition~\ref{p:diagonal-ambi} only asserts that the structure has Calabi
type. In case (i) this is easy: $g=g_+=g_-$ is K\"ahler--Einstein with
$D^g$-parallel product structure, so is the local product of two Riemann
surfaces with constant Gauss curvatures.

In case (ii) $g=g_+$ is K\"ahler--Einstein, Proposition~\ref{p:calabi-type}
implies that the quotient Riemann surface $(\Sigma, g_{\Sigma})$ has constant
Gauss curvature.

In case (iii) $g_\pm$ are extremal, so we have either a local product of two
extremal Riemann surfaces or, in Proposition~\ref{p:calabi-type}, the quotient
Riemann surface $(\Sigma, g_{\Sigma})$ has constant Gauss curvature; it
follows that $g_+$ is locally ambitoric of Calabi type.
\end{proof}

\begin{rem} \label{rem:KE}
The case $Z_+=0$ above yields the following observation of independent
interest: let $(M,g,J,\omega)$ be a K\"ahler--Einstein $4$-manifold with
everywhere degenerate antiselfdual Weyl tensor $W^-$, and trivial first deRham
cohomology group. Then $(M,g,J,\omega)$ admits a globally defined hamiltonian
$2$-form in the sense of \cite{ACG} and, on a dense open subset $M^0$, the
metric is one of the following: a K\"ahler product metric of two Riemann
surfaces of equal constant Gauss curvatures, or a K\"ahler--Einstein metric of
Calabi type, described in Proposition~\ref{p:calabi-type}, or a
K\"ahler--Einstein ambitoric metric of parabolic type (see
section~\ref{summary}).
\end{rem}

\noindent {\it Proof of Theorem}~\ref{thm:AHE-refined}.
For the first part, if $W^+$ and $W^-$ identically vanish, we have a real
space form and $g$ is locally conformally-flat (and is obviously locally
ambitoric).

If $g$ is half-conformally-flat but not flat, then $g$ admits a canonically
defined hermitian structure $J=J_+$, i.e., $g$ is an Einstein, hermitian
self-dual metric (see \cite{AG2} for a classification). In particular, $g$ is
an Einstein metric conformal to a self-dual (or, equivalently, Bochner-flat)
K\"ahler metric $(g_+,J_+)$. We learn from \cite{Bryant,ACG} that such a
K\"ahler metric must be either orthotoric or of Calabi type over a Riemann
surface $(\Sigma, g_{\Sigma})$ of constant Gauss curvature. In both cases the
metric is locally ambitoric by the examples discussed in the previous
subsections.

In the generic case, the result follows from
Propositions~\ref{p:AHE-rough},~\ref{p:AK-bach-flat}
and~\ref{p:generic-ein}. 

The last part follows directly from Proposition~\ref{p:AK-bach-flat}. \qed

\section{Local classification of ambitoric structures} \label{s:loc-class}

To classify ambitoric structures on the dense open set where the (local) torus
action is free (cf.~\cite{Guillemin} for the toric case), let $(M,c,J_+,J_-)$
denote a connected, simply connected, ambihermitian $4$-manifold and
$\Kmap\colon\tor \to C^\infty(M,TM)$ a $2$-dimensional family of pointwise
linearly independent vector fields. Let $\eps\in \Wedge^2\tor^*$ be a fixed
area form.

\subsection{Holomorphic lagrangian torus actions}\label{s:hlag-torus}

We denote by $K_\lambda$ the image of $\lambda\in \tor$ under $\Kmap$, by
$\tor_M$ the rank $2$ subbundle of $TM$ spanned by these vector fields, and by
$\boldsymbol\theta\in \Omega^1(M,\tor)$ the $\tor$-valued $1$-form vanishing
on $\tor_M^\perp\subset TM$ with $\boldsymbol\theta(K_\lambda)=\lambda$.

We first impose the condition that $\Kmap$ is an infinitesimal
$J_\pm$-holomorphic and $\omega_\pm$-isotropic (hence lagrangian) torus
action. We temporarily omit the $\pm$ subscript, since we are studying the
complex structures separately. The lagrangian condition means that $\tor_M$ is
orthogonal and complementary to its image $J\tor_M$ under the complex
structure $J$; thus $J\tor_M=\tor_M^\perp$.  The remaining conditions
(including the integrability of $J$) imply that the vector fields
$\{K_\lambda:\lambda\in\tor\}$ and $\{JK_\lambda:\lambda\in\tor\}$ all commute
under Lie bracket, or equivalently that the dual $1$-forms $\boldsymbol\theta$
and $J\boldsymbol\theta$ are both closed. Thus we may write
$\boldsymbol\theta=\d\ang$ with $\d\d^c\ang=0$, where $\d^c\ang=J\d\ang$ and
the ``angular coordinate'' $\ang\colon M\to \tor$ is defined up to an additive
constant. Conversely, if $\d\d^c\ang=0$ then $\d\ang-\iI \d^c\ang$ generates a
closed differential ideal $\Omega^{(1,0)}$ for $J$ so that $J$ is integrable.

\subsection{Regular ambitoric structures}

We now combine this analysis for the complex structures $J_\pm$. It follows
that $J_+\tor_M$ and $J_-\tor_M$ coincide and that $\tor_M$ is preserved by
the involution $-J_+J_-$.  Since the eigenbundles (pointwise eigenspaces) of
$-J_+J_-$ are $J_\pm$-invariant, $\tor_M$ cannot be an eigenbundle and hence
decomposes into $+1$ and $-1$ eigenbundles $\ximap_M$ and $\etamap_M$: the
line bundles $\ximap_M$, $\etamap_M$, $J_+\ximap_M=J_-\ximap_M$ and
$J_+\etamap_M=J_-\etamap_M$ provide an orthogonal direct sum decomposition of
$TM$.

We denote the images of $\ximap_M$ and $\etamap_M$ under $\d\ang$ by $\ximap$
and $\etamap$ respectively. We thus obtain a smooth map
$(\ximap,\etamap)\colon M\to \Proj(\tor)\times
\Proj(\tor)\setminus\Delta(\tor)$ where $\Delta(\tor)$ is the diagonal.

The derivatives $\d\ximap\in\Omega^1(M,\ximap^*T\Proj(\tor))$ and
$\d\etamap\in\Omega^1(M,\etamap^*T\Proj(\tor))$ vanish on $\tor_M$ (since
$\ximap$ and $\etamap$ are $\tor$-invariant).  In fact, more is true: they
span orthogonal directions in $T^*M$. (Note that $\ximap^*T\Proj(\tor)\cong
\mathrm{Hom}(\ximap,\underline \tor/\ximap)$, with
$\underline\tor:=M\times\tor$, is a line bundle on $M$, and similarly for
$\etamap^*T\Proj(\tor)$.)

\begin{lemma}\label{l:regular} $\d\ximap$ vanishes on $J_\pm\etamap_M$ and
$\d\etamap$ vanishes on $J_\pm\ximap_M$; hence $0=\d\ximap\wedge
\d\etamap\in\Omega^2(M,\ximap^*T\Proj(\tor)\otimes\etamap^*T\Proj(\tor))$ only
on the subset of $M$ where $\d\ximap=0$ or $\d\etamap=0$.
\end{lemma}
\begin{proof} The $1$-form $(J_++J_-)\d\ang$ is closed, vanishes on
$J_\pm\etamap_M$ and $\tor_M$, and takes values in $\ximap\subset
\underline\tor$ (it is nonzero on $J_\pm\ximap_M$). Hence for any section
$\xms$ of $\ximap\subset M\times \tor$,
\begin{equation*}
0=\d\,\bigl( \eps(\xms,(J_++J_-)\d\ang)\bigr)
=\eps(\d\xms\wedge (J_++J_-)\d\ang)
\end{equation*}
and so $(\d\xms\mod \ximap)\wedge (J_++J_-)\d\ang=0$. This implies that
$\d\ximap$ is a multiple $F$ of $\frac12(J_++J_-)\d\ang\in \Omega^1(M,\ximap)$.
Similarly $\d\etamap$ is a multiple $G$ of $\frac12(J_+-J_-)\d\ang$.
\end{proof}
\begin{cor} If $(M,g_\pm,J_\pm,\omega_\pm)$ is ambitoric with
$(\ximap,\etamap)$ as above, then there is a dense open set $M^0$ such that on
each connected component, the ambitoric structure is either of Calabi type, or
$\d\ximap\wedge \d\etamap$ is nonvanishing.
\end{cor}
Indeed, if $\ximap$ and $\etamap$ are functionally dependent on an connected
open set $U$, then one of the two is a constant $[\lambda]\in \Proj(\tor)$ and
$U$ has Calabi type with respect to $K_\lambda$.
\begin{defn}\label{regular} If $\d\ximap\wedge \d\etamap$ vanishes nowhere, we
say $(M,c,J_+,J_-,\boldsymbol K)$ is \emph{regular}.
\end{defn}
In the regular case $\d\ximap=\frac12 F(\ximap)(J_++J_-)\d\ang$ and
$\d\etamap=\frac12 G(\etamap)(J_+-J_-)\d\ang$, where $F$, $G$ are local
sections of $\cO(3)$ over $\Proj(\tor)$; more precisely $F(\ximap)\colon M\to
\mathrm{Hom}(\ximap, \ximap^*T\Proj(\tor))$ and similarly for $G(\etamap)$,
but $T\Proj(\tor)\cong\cO(2)$ using $\eps$. We let $\ximap^\natural$ denote
the composite of $\ximap$ with the natural section of $\cO(1)\otimes\tor$ over
$\Proj(\tor)$, and similarly $\etamap^\natural$. We construct from these
$J_\pm$-related orthogonal $1$-forms
\begin{equation*}
\frac{\d\ximap}{F(\ximap)}, \qquad
\frac{\eps(\d\ang,\etamap^\natural)}{\eps(\ximap^\natural,\etamap^\natural)},
\qquad \frac{\d\etamap}{G(\etamap)},\qquad
\frac{\eps(\ximap^\natural,\d\ang)}{\eps(\ximap^\natural,\etamap^\natural)}
\end{equation*}
(with values in the line bundles $\ximap^*$ or $\etamap^*$) which may be used
to write any $\tor$-invariant metric $g$ in the conformal class as
\begin{equation*}
\frac{\d\ximap^2}{F(\ximap)U(\ximap,\etamap)} +
\frac{\d\etamap^2}{G(\etamap)V(\ximap,\etamap)} +
\frac{F(\ximap)}{U(\ximap,\etamap)}
\biggl(\frac{\eps(\d\ang,\etamap^\natural)}
{\eps(\ximap^\natural,\etamap^\natural)}\biggr)^2 +
\frac{G(\etamap)} {V(\ximap,\etamap)}
\biggl(\frac{\eps(\ximap^\natural,\d\ang)}
{\eps(\ximap^\natural,\etamap^\natural)}\biggr)^2.
\end{equation*}
Here $U$ and $V$ are local sections of $\cO(1,0)$ and $\cO(0,1)$ over
$\Proj(\tor)\times\Proj(\tor)\setminus\Delta(\tor)$.

More explicitly, in a neighbourhood of any point, a basis $\lambda=1,2$ for
$\tor$ may be chosen to provide an affine chart for $\Proj(\tor)$ so that
$K_\xi:=\xi K_1-K_2$ and $K_\eta:=\eta K_1-K_2$ are sections of $\ximap_M$ and
$\etamap_M$ respectively, where $\xi>\eta$ are functionally independent
coordinates on $M$. The components of $\ang\colon M\to \tor$ in this basis
complete a coordinate system $(\xi, \eta, t_1, t_2)$ with coordinate vector
fields
\begin{equation*}
\frac{\partial}{\partial \xi} = \frac{J_+K_\xi}{F(\xi)}, \qquad
\frac{\partial}{\partial\eta} = \frac{J_+K_\eta}{G(\eta)}, \qquad
\frac{\partial}{\partial t_1} = K_1, \qquad \frac{\partial}{\partial t_2} =
K_2.
\end{equation*}
Replacing $(J_+,J_-)$ with $(-J_+,-J_-)$ if necessary, we can assume without
loss that $F$ and $G$ (now functions of one variable) are both positive, and
thus obtain the following description of $\tor$-invariant ambihermitian
metrics in the conformal class.

\begin{lemma}\label{l:ambihermitian-toric}
An ambihermitian metric $(g,J_+,J_-)$ which is regular with respect to a
$2$-dimensional family of commuting, $J_\pm$-holomorphic lagrangian Killing
vector fields is given locally by
\begin{gather}\label{g-eqn}
g = \frac{\d\xi^2}{F(\xi)U(\xi,\eta)} + \frac{\d\eta^2}{G(\eta)V(\xi,\eta)}
+ \frac{F(\xi)(\d t_1 + \eta\, \d t_2)^2}{U(\xi,\eta)\,(\xi-\eta)^2}
+ \frac{G(\eta)(\d t_1 + \xi\, \d t_2)^2}{V(\xi,\eta)\,(\xi-\eta)^2},\\
\label{omega-eqn} \omega_\pm^g =
\frac{\d\xi \wedge (\d t_1 + \eta\,\d t_2)}{U (\xi, \eta)\, (\xi - \eta)} \pm
\frac{\d\eta \wedge (\d t_1 + \xi\,\d t_2)}{V (\xi, \eta)\, (\xi - \eta)},
\displaybreak[0]\\
\label{Jpm}
\d^c_+\xi = \d^c_- \xi = F (\xi)\frac{\d t_1 + \eta \, \d t_2}{\xi-\eta},\quad
\d^c_+\eta = -\d^c_-\eta = G (\eta)\frac{\d t_1 + \xi \, \d t_2}{\xi-\eta}
\end{gather} 
for some positive functions $U$ and $V$ of two variables, and some positive
functions $F$ and $G$ of one variable. \textup(Here and later,
$\d^c_{\pm}h=J_\pm \d h$ for any function $h$.\textup)
\end{lemma}
We now impose the condition that $(c,J_+)$ and $(c,J_-)$ admit
$\tor$-invariant K\"ahler metrics $g_+$ and $g_-$. Let $f$ be the conformal
factor relating $g_\pm$ by $g_- = f^2 g_+$. Clearly $f$ is $\tor$-invariant
and so, therefore, is the metric
\begin{equation*}
g_0 := f \, g_+ = f^{-1} \, g_-
\end{equation*}
which we call the \emph{barycentric metric} of the ambitoric structure.  The
Lee forms, $\theta_\pm^{0}$, of $(g_0,J_\pm)$ are given by $\theta_\pm^0= \mp
\frac{1}{2} \log f$.  Conversely, suppose there is an invariant ambihermitian
metric $g_0$ in the conformal class whose Lee forms $\theta_\pm^0$ satisfy
\begin{gather} \label{mean}
\theta^0_+ + \theta^0_- = 0\\
\label{ck}
d (\theta^0_+ - \theta^0_-) = 0.
\end{gather}
Then writing locally $\theta^0_ + = -\frac 12 \d\log{f} = - \theta^0_-$ for
some positive function $f$, the metrics $g_\pm := f^{\mp1} g_0$ are K\"ahler
with respect to $J_\pm$ respectively.

Thus, regular ambitoric conformal structures are defined by ambihermitian
metrics $g_0$ given locally by Lemma~\ref{l:ambihermitian-toric}, and whose
Lee forms $\theta_\pm^0$ satisfy \eqref{mean} and \eqref{ck}.

\begin{lemma} \label{l:ambitoric} For an ambihermitian metric given by
Lemma~\textup{\ref{l:ambihermitian-toric}} the relation \eqref{mean} is
satisfied \textup(with $g_0=g$\textup) iff $U = U (\xi)$ is independent of
$\eta$ and $V = V(\eta)$ is independent of $\xi$. In this case \eqref{ck} is
equivalent to $U (\xi)^2 = R(\xi)$ and $V(\eta)^2 = R(\eta)$, where $R(s) =
r_0 s^2 + 2 r_1 s + r_2$ is a polynomial of degree at most two.

Under both conditions, the conformal factor $f$ with $g_-=f^2g_+$ is
given---up to a constant multiple---by
\begin{equation}\label{f-gen}
f(\xi,\eta) = \frac{R (\xi)^{1/2} R (\eta)^{1/2} + R (\xi, \eta)}{\xi-\eta}
\end{equation}
where $R (\xi, \eta) =r_0 \xi\eta + r_1 (\xi+\eta) + r_2$ is the
``polarization'' of $R$.
\end{lemma}
\begin{proof} The Lee forms $\theta^g_\pm$ are given by $2\theta^g_\pm =
u_\pm \d\xi + v_\pm \d\eta$, with
\begin{equation*}
u_\pm = \frac{V_{\xi}}{V} \pm \frac{V}{(\xi -\eta)  U}, \qquad
v_\pm = \frac{U_{\eta}}{U}\mp \frac{U}{(\xi -\eta)  V}.
\end{equation*}
In particular, $u_+ + u_- = 2 V_{\xi}/V$ and $v_+ + v_- = 2 U_{\eta}/U$. It
follows that $\theta^g_+ + \theta^g_- = 0$ iff $U_{\eta} = 0$ and $V_{\xi} =
0$. This proves the first part of the lemma.

If (\ref{mean}) is satisfied, then 
\begin{equation*}
\theta^g_+ = \frac 12 \Bigl(
\frac{V (\eta)}{(\xi-\eta) U (\xi)} \, \d\xi
- \frac{U (\xi)}{(\xi - \eta) V (\eta)}\, \d\eta \Bigr).
\end{equation*}
It follows that $\d\theta^g_+ = 0$ iff
\begin{equation}\label{UV-eqn}
2 U^2(\xi)- (\xi-\eta) (U^2)'(\xi) = 2 V^2 (\eta) + (\xi-\eta) (V^2)'(\eta)
\end{equation}
where $U^2(\xi)=U(\xi)^2$ and $V^2(\eta)=V(\eta)^2$. Differentiating twice
with respect to $\xi$, we obtain $(\xi-\eta) (U^2)'''(\xi) = 0$, and similarly
$(\xi-\eta)(V^2)'''(\eta) = 0$. Thus $U^2$ and $V^2$ are both polynomials of
degree at most two. We may now set $\xi=\eta$ in~\eqref{UV-eqn} to conclude
that $U^2$ and $V^2$ coincide. Without loss of generality, we assume that $U$
and $V$ are both positive everywhere, so that $U(\xi) = R (\xi)^{1/2}$ and
$V(\eta) = R(\eta)^{1/2}$ for a polynomial $R$ of degree at most two. By using
the identity
\begin{equation*}
R (\xi) - R (\eta) - \tfrac{1}{2} (\xi - \eta) (R'(\xi) + R'(\eta)) \equiv 0
\end{equation*}
we easily check~\eqref{f-gen}. 
\end{proof}

Note that the quadratic $R$ is, more invariantly, a homogeneous polynomial of
degree $2$ on $\tor$ (an algebraic section of $\cO(2)$ over $\Proj(\tor)$).
However the parametrization of ambitoric structures by $R$ and the local
sections $F$ and $G$ of $\cO(3)$ is not effective because of the SL$(\tor)$
symmetry and homothety freedom in the metric. Modulo this freedom, there are
only three distinct cases for $R$: no real roots ($r_1^2 < r_0r_2$), one real
root ($r_1^2=r_0r_2$) and two real roots ($r_1^2>r_0r_2$). We shall later
refer to these cases as \emph{elliptic, parabolic} and \emph{hyperbolic}
respectively.

\begin{rem}\label{rem-R} The emergence of a homogeneous polynomial of degree
$2$ on $\tor$ merits a more conceptual explanation. It also seems to be
connected with a curious symmetry breaking phenomenon between $\omega_+$ and
$\omega_-$. In~\eqref{omega-eqn}, $\omega_\pm^g$ are interchanged on replacing
$V$ by $-V$. This is compatible with the equality $U^2=V^2$ derived in the
above lemma. However, the choice of square root of $R$ to satisfy positivity
of $g$ breaks this symmetry.
\end{rem}

\subsection{Local classification in adapted coordinates}

The square root in the general form of an ambitoric metric is somewhat
awkward: although we are interested in real riemannian geometry, the complex
analytic continuation of the metric will be branched. This suggests pulling
back the metric to a branched cover and making a coordinate change to
eliminate the square root. This is done by introducing rational functions
$\rho$ and $\sigma$ of degree $2$ such that
\begin{equation}\label{rationalize}
R(\sigma(z))=\rho(z)^2.
\end{equation}
If we then write $\xi=\sigma(x)$, $\eta=\sigma(y)$,
$A(x)=F(\sigma(x))\rho(x)/\sigma'(x)^2$ and
$B(y)=G(\sigma(y))\rho(y)/\sigma'(y)^2$, the barycentric metric may be
rewritten as
\begin{equation}\label{g0-generic}\begin{split}
g_0&=\frac{\d x^2}{A(x)}+\frac{\d y^2}{B(y)}\\
&\quad+A(x)\left(\frac{\sigma'(x)(\d t_1+\sigma(y)\d t_2)}
{(\sigma(x)-\sigma(y))\rho(x)}\right)^2
+B(y)\left(\frac{\sigma'(y)(\d t_1+\sigma(x)\d t_2)}
{(\sigma(x)-\sigma(y))\rho(y)}\right)^2.
\end{split}\end{equation}

There are many solutions to~\eqref{rationalize}. We seek a family that covers
all three cases for $R$ and yields metrics that are amenable to computation.
We do this by solving the equation geometrically.  Let $W$ be a
$2$-dimensional real vector space equipped (for convenience) with a symplectic
form $\kappa\in\Wedge^2 W^*$. Thus we have to do with the geometry of the
projective line $\Proj(W)$ and the representation theory of
$\mathfrak{sl}(W)$, which we summarize in Appendix~\ref{s:plt}
(cf.~\cite{Olver}).  In particular, the space $S^2W^*$ of quadratic forms $p$
on $W$ is a Lie algebra under Poisson bracket $\{,\}$ and has a quadratic form
$p\mapsto Q(p)$ given by the discriminant of $p$; the latter polarizes to give
an inner product $\ipq{p,\tilde p}$ of signature $(2,1)$. For $u\in W$, we
denote by $u^\flat\in W^*$ the linear form $v\mapsto \kappa(u,v)$.

Our construction proceeds by fixing a quadratic form $q\in S^2W^*$. The
Poisson bracket $\{q,\cdot\}\colon S^2W^*\to S^2W^*$ vanishes on the span of
$q$ and its image is the $2$-dimensional subspace $S^2_{0,q}W^*:=q^\perp$. We
thus obtain a map
\begin{equation*}
\mathrm{ad}_q\colon S^2W^*/\spn{q} \to S^2_{\smash{0,q}}W^*.
\end{equation*}
We now define $\sigma_q\colon W\to S^2W^*/\spn{q}$ via the Veronese map
\begin{equation*}
\sigma_q(\bz) = \bz^\flat\otimes\bz^\flat \mod q
\end{equation*}
and let $R_q=\mathrm{ad}_q^* Q$. Thus
$R_q(\sigma_q(\bz))=Q(\{q,\bz^\flat\otimes\bz^\flat\}) =
\ipq{q,\bz^\flat\otimes\bz^\flat}^2$ (see Appendix~\ref{s:plt}~\eqref{Q-ip}
with $p=q$ and $\tilde p =\bz^\flat\otimes\bz^\flat$, which is null) and so
\begin{equation*}
R_q(\sigma_q(\bz))=q(\bz)^2.
\end{equation*}

A geometrical solution to~\eqref{rationalize} is now given by identifying
$\tor$ with $S^2W^*/\spn{q}$, and $R$ with $R_q$. This can have arbitrary type
(elliptic, parabolic or hyperbolic): $R_q$ is positive definite if $Q(q) < 0$,
signature $(1,1)$ if $Q(q) > 0$, or semi-positive degenerate if $Q(q) = 0$.
This geometrical solution represents $\ximap$ as $\sigma_q(\bx)$ and $\etamap$
as $\sigma_q(\by)$, where
\begin{equation*}
(\bx,\by)\colon
M\to \Proj(W)\times\Proj(W)\setminus\Delta(W).
\end{equation*}
For $Q(q)\neq 0$, $\sigma_q$ defines a branched double cover of $\Proj(\tor)$
by $\Proj(W)$. For $Q(q)=0$, the projective transformation appears to be
singular for $q\in \spn{\bz^\flat\otimes \bz^\flat}$, but this singularity is
removable (by sending such $\bz$ to $\spn{\bz^\flat}\odot W^*\mod q$) and
$\sigma_q$ identifies $\Proj(W)$ with $\Proj(\tor)$ via the pencil of lines
through a point on a conic. The following figure illustrates the two cases:
\begin{center}
\includegraphics[width=0.8\textwidth]{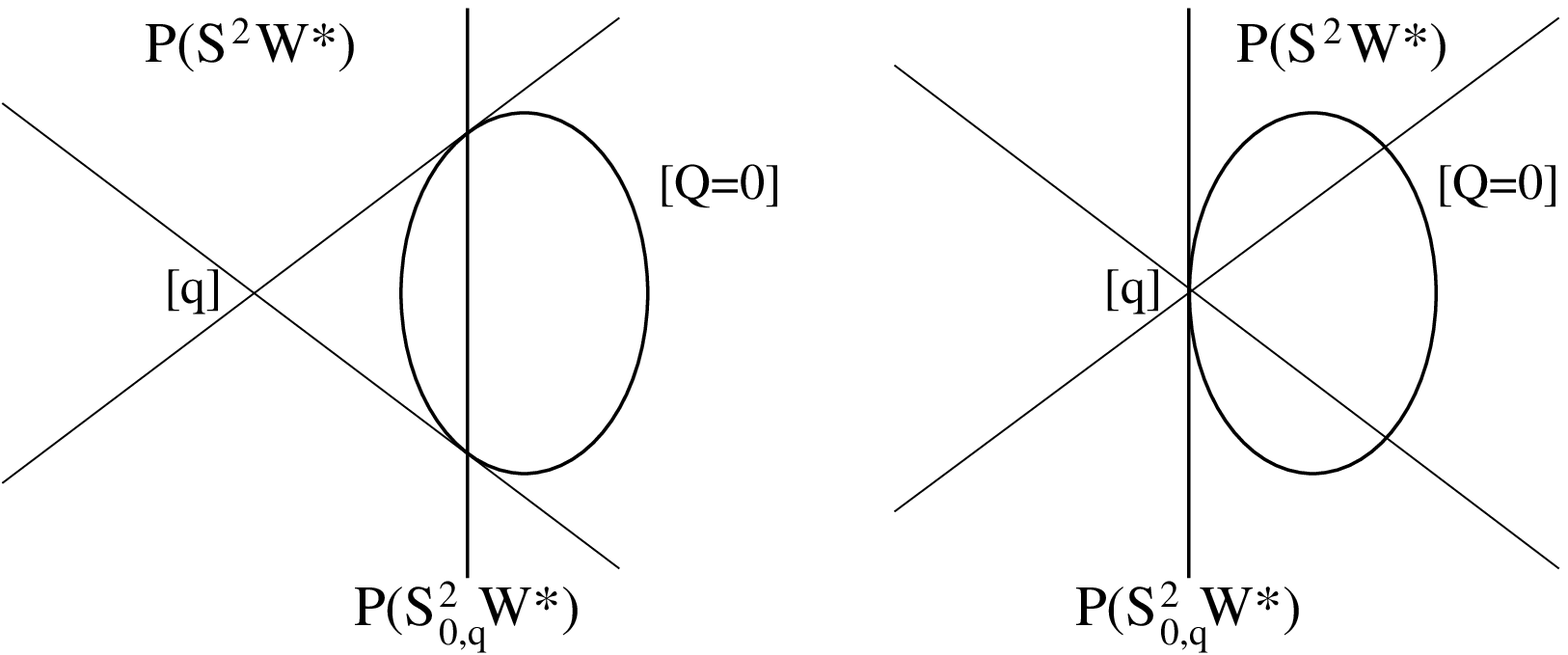}
\end{center}

An area form $\eps\in \Wedge^2\tor^*$ is given by
$\eps(\lambda,\mu)=\ipq{\mathrm{ad}_q\lambda,\mu}$. In particular
\begin{equation*}
\eps(\sigma_q(\bz_1),\sigma_q(\bz_2))
=\ipq{\{q,{\bz_1}^\flat\otimes{\bz_1}^\flat\},
{\bz_2}^\flat\otimes{\bz_2}^\flat}
=2\kappa(\bz_1,\bz_2)q(\bz_1,\bz_2),
\end{equation*}
where $q(\bz_1,\bz_2)$ is the symmetric bilinear form obtained by
polarization. It follows that the barycentric metric $g_0$ may be written
invariantly as
\begin{equation*}
\frac{\d\bx^2}{A(\bx)} + \frac{\d\by^2}{B(\by)}
+ A(\bx) \left(\frac{\ipq{\d\taumap,\by\otimes\by}}
{\kappa(\bx,\by)q(\bx,\by)} \right)^2
+ B(\by) \left(\frac{\ipq{\d\taumap,\bx\otimes\bx}}
{\kappa(\bx,\by)q(\bx,\by)} \right)^2,
\end{equation*}
where $A,B$ are local sections of $\cO(4)$ over $\Proj(W)$,
$\d\taumap=\frac12\{q,\d\ang\}$, and we omit to mention use of the natural
lift $(\cdot)^\natural$ to $\cO(1)\otimes W$ over $\Proj(W)$. Note that
$\ipq{q,\d\taumap}=0$.

A more concrete expression may be obtained by introducing a symplectic basis
$e_1, e_2$ of $W$ (so that $\kappa(e_1,e_2)=1$) and hence an affine coordinate
$z$ on $\Proj(W)$: see Appendix~\ref{s:plt}. In particular, $\kappa(x
e_1+e_2,y e_1+e_2)=x-y$ and any quadratic form $p\in S^2W^*$ may be
written
\begin{equation*}
p(z) = p_0 z^2 + 2 p_1 z + p_2
\end{equation*}
with polarization given by
\begin{equation*}
p(x,y) = p_0 xy + p_1(x+y)+p_2.
\end{equation*}
Elements of $\tor$ may thus represented by triples $[w]=[w_0,w_1,w_2]\in
S^2W^*/\spn{q}$, or by the corresponding elements $p=(p_0,p_1,p_2)$ of
$S^2_{0,q}W^*$ where $p=\frac12\{q,w\}$. The corresponding vector field on $M$
will be denoted $K^{[w]}$ or $K^{(p)}$, so that $\d\ang(K^{[w]})=[w]$ and
$\d\taumap(K^{(p)})=p$. (The factor $1/2$ in the formula
$\d\taumap=\frac12\{q,\d\ang\}$ is a convenience.)

\begin{thm}\label{thm:ambitoric} Let $(M,c,J_+,J_-,\tor)$ be an ambitoric
$4$-manifold with barycentric metric $g_0$ and K\"ahler metrics
$(g_+,\omega_+)$ and $(g_-,\omega_-)$. Then, about any point in a dense open
subset of $M$, there is a neighbourhood in which $(c,J_+,J_-)$ is either of
Calabi type with respect to some $\lambda\in\tor$, or there there are
$\tor$-invariant functions $x,y$, a quadratic polynomial
$q(z)=q_0z^2+2q_1z+q_2$, and functions $A(z)$ and $B(z)$ of one variable with
respect to which\textup:
\begin{gather} \label{g0-xy} \begin{split} 
g_0 & = \frac{\d x^2}{A(x)} + \frac{\d y^2}{B(y)}\\ & \quad
+A(x) \left(\frac{y^2 \d\tau_0 + 2y \d\tau_1 + \d\tau_2}{(x-y)q(x,y)}\right)^2
+B(y) \left(\frac{x^2 \d\tau_0 + 2x \d\tau_1 + \d\tau_2}{(x-y)q(x,y)}\right)^2,
\end{split}\displaybreak[0]\\
\label{omegaplus-xy} \begin{split} 
\omega_+ &= \frac{x-y}{q(x,y)}
\left(\frac{\d x\wedge\d^c_+x}{A(x)} + \frac{\d y\wedge\d^c_+y}{B(y)}\right)\\
&= \frac{\d x\wedge (y^2 \d\tau_0 + 2y \d\tau_1 + \d\tau_2)
+ \d y \wedge (x^2 \d\tau_0 + 2x \d\tau_1 + \d\tau_2)}{q(x,y)^2},
\end{split}\\
\label{omegaminus-xy} \begin{split} 
\omega_- & = \frac{q(x,y)}{x-y}
\left(\frac{\d x\wedge\d^c_-x}{A(x)} + \frac{\d y\wedge\d^c_-y}{B(y)}\right)\\
&= \frac{\d x \wedge (y^2 \d\tau_0 + 2 y \d\tau_1 + \d\tau_2)
-  \d y \wedge (x^2 \d\tau_0 + 2 x \d\tau_1 + \d\tau_2)}{(x-y)^2}.
\end{split}
\end{gather}
where $2q_1\d\tau_1=q_0\d\tau_2+q_2\d\tau_0$ and $q(x,y)=q_0xy+q_1(x+y)+q_2$.

Conversely, for any data as above, the above metric and K\"ahler forms do
define an ambitoric K\"ahler structure on any simply connected open set where
$\omega_\pm$ are nondegenerate and $g_0$ is positive definite.
\end{thm}
\begin{proof} The fact that regular ambitoric conformal structures have this
form follows easily from Lemmas~\ref{l:ambihermitian-toric}
and~\ref{l:ambitoric}. One can either substitute into the invariant form of
the metric, or carry out the coordinate transformation explicitly
using~\eqref{g0-generic}. We deduce from \eqref{Jpm} that
\begin{equation} \label{dc-xy} \begin{split}
\d^c_+ x =\d^c_- x &= \frac{A(x)}
{(x - y) q(x,y)}(y^2 \d\tau_0 + 2 y \d\tau_1 + \d\tau_2), \\
\d^c_+ y =-\d^c_- y&= \frac{B(y)}
{(x - y) q(x,y)}(x^2 \d\tau_0 + 2 x \d\tau_1 + \d\tau_2).
\end{split} \end{equation}
The computation of the conformal factor
\begin{equation} \label{fxy} 
f(x,y) = \frac{q(x,y)}{x - y}
\end{equation}
with $\omega_-=f^2\omega_+$ requires more work, but it is straightforward to
check that $\omega_\pm$ are closed, and hence deduce conversely that any
metric of this form is ambitoric.
\end{proof}

\begin{defn}\label{d:rough} A regular ambitoric structure~\eqref{g0-xy} is
said to be of \emph{elliptic}, \emph{parabolic}, or \emph{hyperbolic} type if
the number of distinct real roots of $q(z)$ (on $\Proj(W)$) is zero, one or
two respectively.
\end{defn}

For later use, we compute the momentum maps $\mu^\pm$ (as functions of $x$ and
$y$) of the (local) toric action with respect to $\omega_{\pm}$. Since
\begin{equation*}
\omega_- = -\d\chi, \qquad
\chi= \frac{xy\, \d\tau_0 + (x+y)\d\tau_1 + \d\tau_2}{x-y},
\end{equation*}
we have, for any $p\in S^2_{0,q}W^*$ (so $2p_1q_1=p_0q_2+p_2q_0$) and any
$c\in\R$, a Killing potential
\begin{equation}\label{mu-minus}
\mu^-_{p,c}= -\frac{p(x,y)+c(x-y)}{x-y} =-\frac{p_0 x y + p_1 (x + y) + p_2+c
(x - y)}{x-y}
\end{equation}
for $K^{(p)}$. (Invariantly, this is the contraction of $p+c\kappa$ with
$-\bx\otimes\by/\kappa(\bx,\by)$.)

For $\omega_+$, we use the fact that $\d\taumap=\frac12\{q,\d\ang\}$ and
compute, for any $w\in S^2W^*$,that
\begin{equation*}
\iota_{K^{[w]}}\omega_+
= \frac{\frac12\{q,w\}(y)\,\d x+\frac12\{q,w\}(x)\,\d y}{q(x,y)},
\end{equation*}
and so
\begin{equation}\label{mu-plus}
\mu^+_{w}=-\frac{w(x,y)}{q(x,y)}
=-\frac{w_0 x y + w_1 (x + y) + w_2}{q_0 x y + q_1 (x + y) + q_2}
\end{equation}
(the contraction of $w$ with $-\bx\otimes \by/q(\bx,\by)$) is a
Killing potential for $K^{[w]}$.

\section{Extremal and conformally Einstein ambitoric surfaces}
\label{s:curvature}

We now compute the Ricci forms and scalar curvatures of a regular ambitoric
K\"ahler surface (cf.~\cite{Abreu0} for the toric case), and hence give a
local classification of extremal ambitoric structures. By considering the Bach
tensor, we also identify the regular ambitoric structures which are
conformally Einstein.

\subsection{Ricci forms and scalar curvatures}

As in~\cite{Cal1,ACG}, we adopt a standard method for computing the Ricci form
of a K\"ahler metric as the curvature of the connection on the canonical
bundle: the log ratio of the symplectic volume to any holomorphic volume is a
Ricci potential. For regular ambitoric metrics, $\d\ang+\iI \d^c_{\pm}\ang$ is
a $J_\pm$-holomorphic $\tor$-valued $1$-form. From~\eqref{dc-xy} we obtain
that for any $w\in S^2W^*$,
\begin{equation*}
\ipq{\d^c_\pm \taumap,w}
= \frac{\{q,w\}(x)}{A(x)} \d x \mp \frac{\{q,w\}(y)}{B(y)} \d y
\end{equation*}
(since $\ipq{q,\d\taumap}=0$), and deduce (using
$\d\taumap=\frac12\{q,\d\ang\}$) that for any $p\in S^2_{0,q}W^*$,
\begin{equation*}
\ipq{\d^c_\pm \ang,p}
= -\frac{p(x)}{A(x)} \d x \pm \frac{p(y)}{B(y)} \d y.
\end{equation*}
Using an arbitrary basis for $S^2_{0,q}W^*$ we find that
\begin{equation*}
v_0 = \frac{(x-y)^2 q(x,y)^2}
{A(x)^2 B(y)^2} \d x \wedge \d^c_+ x \wedge \d y \wedge \d^c_+ y.
\end{equation*}
can be taken as a holomorphic volume for both $J_+$ and $J_-$ (up to
sign). The symplectic volumes $v_\pm$ of $\omega_\pm$ are
\begin{align*}
v_+ &= \frac{(x - y)^2}{q (x, y)^2 A (x) B (y)}
\d x \wedge \d^c_+ x \wedge \d y \wedge \d^c_+ y,\\
v_- &= \frac{q (x, y)^2}{(x - y)^2 A (x) B (y)}
\d x \wedge \d^c_- x \wedge \d y \wedge \d^c_- y.
\end{align*}
Hence the Ricci forms $\rho_\pm= - \frac{1}{2} \d\d^c_\pm \log |v_\pm/v_0|$ of
$\omega_\pm$ are given by
\begin{equation*}
\rho_+ = - \tfrac{1}{2} \d\d^c_+ \log\frac{A(x)B(y)}{q(x,y)^4}, \qquad
\rho_- = - \tfrac{1}{2} \d\d^c_- \log\frac{A(x)B(y)}{(x-y)^4}.
\end{equation*}
The $2$-forms $\d\d^c x$ and $\d\d^c y$ are obtained by
differentiating the two sides of \eqref{dc-xy}. After some work,
we obtain
\begin{gather*} \begin{split}
\d\d^c_\pm x&= \left(A'(x) - \frac{q(x)-q_0 \, (x-y)^2}
{(x-y) q(x,y)}A(x)\right)\frac{\d x \wedge \d^c_\pm x}{A(x)} \\
&\quad\pm \frac{q (y)A(x)}{(x-y)q(x,y)} \frac{\d y\wedge \d^c_\pm y}{B(y)},
\end{split}\displaybreak[0]\\ 
\begin{split} 
\d\d^c_\pm y
&= \mp\frac{q(x)B(y)}{(x-y)q(x,y)}\frac{\d x\wedge \d^c_\pm x}{A(x)}\\
&\quad+ \left(B'(y) + \frac{q(y)-q_0 \, (x-y)^2}
{(x - y) q (x,y)}B(y)\right) \frac{\d y \wedge \d^c_\pm y}{B(y)}.
\end{split}\end{gather*}
Hence for any $\tor$-invariant function $\phi=\phi(x,y)$,
\begin{align*}
\d\d^c_\pm\phi
&=\phi_{xx}\,\d x\wedge\d^c_\pm x + \phi_{yy}\,\d y\wedge\d^c_\pm y
+ \phi_{xy}(\d x\wedge \d^c_\pm y+\d y\wedge \d^c_\pm x)\\
&\quad+ \phi_x\, \d\d^c_\pm x + \phi_y\, \d\d^c_\pm y\displaybreak[0]\\
&= \left(\bigl(A(x)\phi_x\bigr)_x - \frac{q(x)-q_0 \, (x-y)^2}
{(x-y) q(x,y)}A(x)\phi_x\right)\frac{\d x \wedge \d^c_\pm x}{A(x)} \\
&\quad\pm \frac{q(y)A(x)\phi_x}{(x-y)q(x,y)} \frac{\d y\wedge \d^c_\pm y}{B(y)}
\mp \frac{q(x)B(y)\phi_y}{(x-y)q(x,y)}\frac{\d x\wedge \d^c_\pm x}{A(x)}\\
&\quad+ \left(\bigl(B(y)\phi_y\bigr)_y + \frac{q(y)-q_0 \, (x-y)^2}
{(x-y) q(x,y)}B(y)\phi_y\right) \frac{\d y \wedge \d^c_\pm y}{B(y)}\\
&\quad+\phi_{xy}(\d x\wedge \d^c_\pm y+\d y\wedge \d^c_\pm x).
\end{align*}
In particular, the expression is both $J_+$ and $J_-$ invariant iff
$\phi_{xy}=0$.  The invariant part simplifies considerably when expressed in
terms of the K\"ahler forms $\omega_{\pm}^0$ of the barycentric metric. Using
the fact that $q_0 x + q_1$ and $q_0 y+q_1$ are the $y$ and $x$ derivatives of
$q(x,y)$ respectively, we eventually obtain
\begin{align*}
\d\d^c_\pm\phi
&=\frac{q(x,y)^2}{2}\left(\biggl[\frac{A(x)\phi_x}{q(x,y)^2}\biggr]_x
\pm\biggl[\frac{B(y)\phi_y}{q(x,y)^2}\biggr]_y\right)\omega_+^0\\
&\quad+\frac{(x-y)^2}{2}\left(\biggl[\frac{A(x)\phi_x}{(x-y)^2}\biggr]_x
\mp\biggl[\frac{B(y)\phi_y}{(x-y)^2}\biggr]_y\right)\omega_-^0\\
&\quad+\phi_{xy}(\d x\wedge \d^c_\pm y+\d y\wedge \d^c_\pm x).
\end{align*}
Substituting the Ricci potentials for $\phi$, we thus obtain, after a little
manipulation,
\begin{gather*}\begin{split}
\rho_+ &=-\frac{q(x,y)^2}{4}\left(
\biggl[q(x,y)^2\Bigl[\frac{A(x)}{q(x,y)^4}\Bigr]_x\biggr]_x
+\biggl[q(x,y)^2\Bigl[\frac{B(y)}{q(x,y)^4}\Bigr]_y\biggr]_y
\right)\omega_+^0\\
&\quad-\frac{(x-y)^2}{4}\left(
\biggl[\frac{q(x,y)^4}{(x-y)^2}\Bigl[\frac{A(x)}{q(x,y)^4}\Bigr]_x\biggr]_x
-\biggl[\frac{q(x,y)^4}{(x-y)^2}\Bigl[\frac{B(y)}{q(x,y)^4}\Bigr]_y\biggr]_y
\right)\omega_-^0\\
&\quad+2\frac{(q_0q_2-q_1^2)(\d x\wedge\d^c_+ y+\d y\wedge\d^c_+ x)}{q(x,y)^2}
\end{split},\displaybreak[0]\\
\begin{split}
\rho_- &=-\frac{q(x,y)^2}{4}\left(
\biggl[\frac{(x-y)^4}{q(x,y)^2}\Bigl[\frac{A(x)}{(x-y)^4}\Bigr]_x\biggr]_x
-\biggl[\frac{(x-y)^4}{q(x,y)^2}\Bigl[\frac{B(y)}{(x-y)^4}\Bigr]_y\biggr]_y
\right)\omega_+^0\\
&\quad-\frac{(x-y)^2}{4}\left(
\biggl[(x-y)^2\Bigl[\frac{A(x)}{(x-y)^4}\Bigr]_x\biggr]_x
+\biggl[(x-y)^2\Bigl[\frac{B(y)}{(x-y)^4}\Bigr]_y\biggr]_y
\right)\omega_-^0\\
&\quad+2\frac{\d x\wedge \d^c_- y+\d y\wedge \d^c_- x}{(x-y)^2}.
\end{split}
\end{gather*}
(In particular $g_+$ can only be K\"ahler--Einstein in the parabolic
case---when $q$ has a repeated root---while $g_-$ is never
K\"ahler--Einstein.) The scalar curvatures, given by
$s_\pm=2\rho_\pm\wedge\omega_\pm/v_\pm$, should be SL$(W)$-invariants of $A,B$
and $q$. For this we observe that for any quadratic form $p$ with $Q(p)=0$,
and any function $A$ of one variable,
\begin{equation*}
p(x)^2\left( \biggl[p(x)\Bigl[\frac{A(x)}{p(x)^2}\Bigr]_x\biggr]_x\right)
=p(x)A''(x)-3 p'(x) A'(x)+6 p''(x) A(x),
\end{equation*}
which is the transvectant $(p,A)^{(2)}$ when $A$ is a quartic (or more
generally, a local section of $\cO(4)$)---see Appendix~\ref{s:plt}. We apply
this with $p(x)=q(x,y)^2$ and $p(x)=(x-y)^2$, and treat $B(y)$ in a similar
way to obtain,
\begin{align}\label{splus}
s_+ &= -\frac{(q(x,y)^2,A(x))^{(2)}+(q(x,y)^2,B(y))^{(2)}}{(x-y) q(x,y)}\\
\label{sminus}
s_- &=-\frac{((x-y)^2,A(x))^{(2)}+((x-y)^2,B(y))^{(2)}}{(x - y)q (x, y)},
\end{align}
where $y$ is fixed when taking a transvectant with respect to $x$ and vice
versa.

\subsection{Extremality and Bach-flatness}\label{s:ext-bflat}

The K\"ahler metrics $g_\pm$ are extremal if their scalar curvatures $s_\pm$
are Killing potentials. Since the latter are $\tor$-invariant (and $\tor_M$ is
lagrangian), this can only happen if $s_\pm$ is the momentum of some Killing
vector field $K^{(p)}\in \tor$. The condition is straightforward to solve for
$g_+$: equating~\eqref{splus} (for $s_+$) and~\eqref{mu-plus} (for
$\mu^+$) yields
\begin{equation}\label{extremalplus}
(q(x,y)^2, A(x))^{(2)} +(q(x,y)^2, B(y))^{(2)} = (x-y) w(x,y).
\end{equation}
Differentiating three times with respect to $x$ or three times with respect to
$y$ shows that $A$ and $B$ (respectively) are polynomials of degree at most
four. We now introduce polynomials $\Pi$ and $P$ determined by $A=\Pi+P$ and
$B=\Pi-P$. Since the left hand side of~\eqref{extremalplus} is antisymmetric
in $(x,y)$, the symmetric part of the equation is
\begin{equation}\label{plus-sym}
(q(x,y)^2, \Pi(x))^{(2)} +(q(x,y)^2, \Pi(y))^{(2)} =0
\end{equation}
On restriction to the diagonal ($x=y$) in this polynomial equation, we obtain
\begin{equation*}
q^2 \Pi'' - 3 q q' \Pi' + 3 (q')^2 \Pi = 0. 
\end{equation*}
To solve this linear ODE for $\Pi$, we set $\Pi(z)=q(z)\pi(z)$ to get $q^2
(q'' \pi - q' \pi' + q \pi'') = 0$, from which we deduce that $\pi$ is a
polynomial of degree $\leq 2$ ($\pi'''=0$) and that $\pi$ is orthogonal to
$q$. Conversely, by straightforward verification, this ensures $\Pi$
solves~\eqref{plus-sym}.

The antisymmetric part of~\eqref{extremalplus} is
\begin{equation*}
(q(x,y)^2, P(x))^{(2)} - (q(x,y)^2, P(y))^{(2)}
= (x - y) (w_0 x y + w_1 (x + y) + w_2).
\end{equation*}
The left hand side is clearly divisible by $x-y$ and since it is quadratic in
both $x$ and $y$, the quotient is (affine) linear in both $x$ and $y$, hence
the polarization of a quadratic form. To compute this quadratic form we divide
the left hand side by $x-y$ and restrict to the diagonal to obtain
\begin{equation*}
q^2 P''' - 3 q q' P'' + 3\bigl((q')^2 + q q''\bigr) P' - 6 q' q'' P\\
= \{q,(q,P)^{(2)}\}
\end{equation*}
As $q$ is nonzero, any quadratic form may be represented as $(q,P)^{(2)}$ for
some quartic $P$, and hence any quadratic form $w$ orthogonal to $q$ has the
form $w=\{q,(q,P)^{(2)}\}$ for some quartic $P$. Thus
\begin{equation*}
s_+=-\frac{w(x,y)}{q(x,y)},
\end{equation*}
where $w=\{q,(q,P)^{(2)}\}$ is orthogonal to $q$. Hence, except in the
parabolic case ($q$ degenerate), $s_+$ is constant iff it is identically zero.
\smallbreak

Remarkably, the extremality condition for $g_-$ coincides with that for
$g_+$. To see this, we equate~\eqref{sminus} (for $s_-$) and~\eqref{mu-minus}
(for $\mu^-$) to obtain the extremality equation
\begin{equation}\label{extremalminus}
((x-y)^2, A(x))^{(2)} + ((x-y)^2, B(y))^{(2)}
= q(x,y)(p_0 x y + p_1 (x + y) + p_2 + c (x-y)),
\end{equation}
which we shall again decompose into symmetric and antisymmetric parts: for
this we first observe, by taking three derivatives, that $A$ and $B$ are
polynomials of degree $\leq 4$, we write $A=\Pi+P$, $B=\Pi-P$ as before.

The symmetric part, namely
\begin{equation*}
((x-y)^2,\Pi(x))^{(2)} + ((x-y)^2,\Pi(y))^{(2)}
= q(x,y) (p_0 xy + p_1 (x + y) + p_2),
\end{equation*}
immediately yields, on restricting to the diagonal ($y=x$),
$\Pi(z)=q(z)\pi(z)$ with $\pi(z)=p(z)/24$. Further, since $\ip{p,q}=0$, the
equation is satisfied with this Ansatz.

The antisymmetric part, namely
\begin{equation*}
((x-y)^2,P(x))^{(2)} - ((x-y)^2,P(y))^{(2)}
=c q(x,y)(x-y)
\end{equation*}
yields $c=0$ (divide by $x-y$ and restrict to the diagonal) and is then
satisfied identically for \emph{any} polynomial $P$ of degree $\leq 4$. Thus we
again have an extremal K\"ahler metric with
\begin{equation*}
s_- = -\frac{24 \pi(x,y)}{x-y}.
\end{equation*}
Note that $s_-$ is constant iff it is identically zero.

The Bach-flatness condition is readily found using Lemma~\ref{l:bach-flat}:
since $-v_-/v_+=q(x,y)^4/(x-y)^4$, equation~\eqref{bflat} holds iff $\pi(x,y)$
and $w(x,y)$ are linearly dependent.

\begin{thm}\label{thm:main} Let $(J_+,J_-, g_+,g_-,\tor)$ be a regular
ambitoric structure as in Theorem~\textup{\ref{thm:ambitoric}}. Then
$(g_+,J_+)$ is an extremal K\"ahler metric if and only if $(g_-,J_-)$ is an
extremal K\"ahler metric if and only if
\begin{equation}\begin{split}
A(z)&=q(z)\pi(z)+P(z),\\
B(z)&=q(z)\pi(z)-P(z),\\
\end{split}\end{equation}
where $\pi(z)$ is a polynomial of degree at most two orthogonal to $q(z)$ and
$P(z)$ is polynomial of degree at most four. The conformal structure is
Bach-flat if and only if the quadratic polynomials $\pi$ and
$\{q,(q,P)^{(2)}\}$ are linearly dependent.
\end{thm}

\subsection{Compatible metrics with diagonal Ricci tensor}
\label{s:diagonal-ricci}

A consequence of the explicit form~\eqref{mu-minus}--\eqref{mu-plus} for the
Killing potentials is that any regular ambitoric structure admits
$\tor$-invariant compatible metrics with diagonal Ricci tensor.

\begin{prop}\label{p:diagonal-ricci} Let $(g_\pm,J_\pm,\omega_\pm,\tor)$ be a
regular ambitoric structure as in Theorem~\textup{\ref{thm:ambitoric}}. Then
for any quadratic $p(z)=p_0 z^2+2p_1 z+p_2$ orthogonal to $q$,
\begin{equation*}
g = \frac{(x-y)^2}{p(x,y)^2} g_- = \frac{q(x,y)^2}{p(x,y)^2} g_+
= \frac{q(x,y)(x-y)}{p(x,y)^2}g_0
\end{equation*}
has diagonal Ricci tensor and scalar curvature
\begin{align*}
s^g &= -\frac{(p(x,y)^2,A(x))^{(2)}+(p(x,y)^2,B(y))^{(2)}}{(x-y) q(x,y)}
\end{align*}
Any $\tor$-invariant compatible metric with diagonal Ricci tensor arises in
this way.
\end{prop}
\begin{proof} By Proposition~\ref{p:J-inv-ric}, a compatible metric
$g=\varphi_+^{-2}g_+= \varphi_-^{-2}g_-$ has diagonal Ricci tensor iff
$\varphi_\pm$ are Killing potentials with respect to $\omega_\pm$. For $g$ to
be $\tor$-invariant, the corresponding Killing fields must be in $\tor$, hence
$\varphi_+=w(x,y)/q(x,y)$ for some $w\in S^2W^*$ and
$\varphi_-=p(x,y)/(x-y)+c$ for some $p\in S^2_{0,q}W^*$. The equality
$\varphi_+^{-2}g_+= \varphi_-^{-2}g_-$ is satisfied iff $w=p$ and $c=0$. The
formula for the scalar curvature is a tedious computation which we omit.
\end{proof}
We have seen in Theorem~\ref{thm:AT-rough} that the riemannian analogues of
Pleba\'nski--Demia\'nski metrics are compatible CSC metrics compatible with
diagonal Ricci tensor. Since the scalar curvature of $g$ has the same form as
the scalar curvature of $g_+$ (with $q$ replaced by $p$), the calculations
used for the extremality of $g_+$ establish the following result.

\begin{thm} \label{thm:einstein-maxwell} A compatible metric $g$ with
diagonal Ricci tensor is CSC if and only if
\begin{equation}\begin{split}
A(z)&=p(z)\rho(z)+R(z),\\ B(z)&=p(z)\rho(z)-R(z),\\
\end{split}\end{equation}
where $\rho(z)$ is a quadratic polynomial orthogonal to $p(z)$ and $R(z)$ is a
quartic polynomial orthogonal to $q(z)p(z)$ \textup(equivalently $(q,R)^{(2)}$
is orthogonal to $p$ or, equally, $(p, R)^{(2)}$ is orthogonal to
$q$\textup). The metric is Einstein when $\rho(z)$ is a multiple of $q(z)$.
\end{thm}

This is strikingly similar to, yet also different from, the extremal case.
They overlap in the Einstein case, and in the parabolic case with $p$ a
multiple of $q$.

\subsection{Normal forms}\label{summary}

The projective choice of coordinate on $\Proj(W)$ can be used to set $q(z)=1$,
$z$ or $1+z^2$ in the parabolic, hyperbolic or elliptic cases respectively.
To describe the curvature conditions in these normal forms, we write $A(z) =
a_0 z^4 + a_1 z^3 + a_2 z^2 + a_3 z + a_4$ and $B(z) = b_0 z^4 + b_1 z^3 + b_2
z^2 + b_3 z + b_4$.

\subsubsection*{Parabolic type} When $q(z)=1$, $\d\tau_0=0$, and $S^2_{0,q}W^*=
\{p(z)=2p_1z+p_2\}$; we may represent $[w]\in S^2W^*/\spn{q}$ by
$w_0z^2+2w_1z$ with $\frac12\{q,w\}=-w_0 z-w_1$ and define components of
$\xi\in \tor^*$ by $\xi(p)=2\xi_1 p_1+\xi_2 p_2$. Modulo constants, the
Killing potentials for $\omega_\pm$ are spanned by
\begin{align*}
\mu^+_1 &= x + y, & \mu^+_2 &= x y, \\
\mu^-_1 &= - \frac{1}{x - y}, &\mu^-_2 &= - \frac{x + y}{2 (x - y)},
\end{align*}
while the barycentric metric $g_0$ and K\"ahler forms $\omega_\pm$ take
the form
\begin{align*}
g_0 &= \frac{\d x^2}{A(x)}   + \frac{\d y^2}{B(y)}
+ \frac{A(x)(\d t_1 + y\, \d t_2)^2}{(x-y)^2}
+ \frac{B(y)(\d t_1 + x\, \d t_2)^2}{(x-y)^2},\\
\omega_+ &= \d x \wedge (\d t_1 + y \, \d t_2)
+ \d y \wedge (\d t_1 + x \, \d t_2),\\
\omega_- &= \frac{\d x \wedge (\d t_1 + y \, \d t_2)}{(x-y)^2} - 
\frac{\d y \wedge (\d t_1 + x \, \d t_2)}{(x - y)^2}.
\end{align*}
The metrics $g_\pm$ are extremal iff
\begin{equation*}
a_0 + b_0 = a_1 + b_1 = a_2 + b_2 = 0,
\end{equation*}
in which case
\begin{equation*}
s_+ =  - 6 a_1 - 12 a_0 \, \mu^+_1,\qquad
s_- = 12 (a_4 + b_4) \, \mu^-_1 + 12 (a_3 + b_3) \, \mu^-_2.
\end{equation*}
The structure is Bach-flat iff $a_1 + 4a_0 z$ and $-(a_4 + b_4) + (a_3 +
b_3)z$ are linearly dependent, i.e.,
\begin{equation*}
a_1 (a_3 + b_3) + 4 a_0 (a_4 + b_4) =0.
\end{equation*}
For $p(z)=z$, $g=q(x,y)^2g_+/p(x,y)^2$ is CSC iff $a_0+b_0=a_2+b_2=a_4+b_4=0$,
and $a_1=b_1$.

\subsubsection*{Hyperbolic type} When $q(z)=2z$, $\d\tau_1=0$, and
$S^2_{0,q}W^*= \{p(z)=p_0z^2+p_2\}$; we may represent $[w]\in S^2W^*/\spn{q}$
by $w_0z^2+w_2$ with $\frac12\{q,w\}= -w_0z^2+w_2$ and define components of
$\xi\in \tor^*$ by $\xi(p)=\xi_1 p_2+\xi_2 p_0$. Modulo constants, the
Killing potentials for $\omega_\pm$ are spanned by
\begin{align*}
\mu^+_1 &= - \frac{1}{x + y}, &\mu^+_2 &= \frac{xy}{x + y}, \\
\mu^-_1 &= - \frac{1}{x - y}, &\mu^-_2 &=-\frac{xy}{x - y},
\end{align*}
while the barycentric metric $g_0$ and K\"ahler forms $\omega_\pm$ then take
the form
\begin{align*}
g_0 &= \frac{\d x^2}{A(x)} + \frac{\d y^2}{B(y)}
+ \frac{A(x)(\d t_1 + y^2 \d t_2)^2}{(x^2-y^2)^2}
+ \frac{B(y)(\d t_1 + x^2 \d t_2)^2}{(x^2-y^2)^2}\\
\omega_+  &= \frac{\d x\wedge (\d t_1 + y^2 \, \d t_2)}{(x + y)^2}
+ \frac{\d y \wedge (\d t_1 + x^2 \, \d t_2)}{(x + y)^2}\\
\omega_- &=  \frac{\d x\wedge (\d t_1 + y^2 \, \d t_2)}{(x - y)^2}
- \frac{\d y \wedge (\d t_1 + x^2 \, \d t_2)}{(x - y)^2}.
\end{align*}
The metrics $g_\pm$ are extremal iff
\begin{equation*}
a_0 + b_0 = a_2 + b_2 = a_4 + b_4 = 0,
\end{equation*}
in which case
\begin{equation*}
s_\pm = -6 (a_3 \pm  b_3) \, \mu^\pm_1 - 6 (a_1 \pm  b_1) \, \mu^\pm_2.
\end{equation*}
The Bach-flatness condition is therefore
\begin{equation*}
(a_3 - b_3) (a_1 + b_1) + (a_3 + b_3) (a_1 - b_1) = 0.
\end{equation*}
For $p(z)=1+\eps z^2$, $g=q(x,y)^2g_+/p(x,y)^2$ is CSC iff
$a_0+b_0=-\eps^2(a_4+b_4)$, $a_1+b_1=\eps(a_3+b_3)$, $a_2+b_2=0$, and
$a_1-b_1=-\eps(a_3-b_3)$. The resulting family
\begin{equation*}
\frac1{(1+\eps xy)^2}\biggl(\frac{(x^2-y^2)\d x^2}{A(x)} +
\frac{(x^2-y^2)\d y^2}{B(y)} + \frac{A(x)(\d t_1+y^2 \d t_2)^2}{x^2-y^2} +
\frac{B(y)(\d t_1+x^2 \d t_2)^2}{x^2-y^2}\biggr)
\end{equation*}
of metrics, where
\begin{align*}
A(z) &= h+\kappa + (\sigma+\delta)z + \gamma z^2 + \eps(\sigma-\delta)z^3
+(\lambda-\eps^2 h)z^4,\\
B(z) &= h-\kappa + (\sigma-\delta)z - \gamma z^2 + \eps(\sigma+\delta)z^3
-(\lambda+\eps^2 h)z^4,
\end{align*}
is an analytic continuation of the Pleba\'nski--Demia\'nski
family~\cite{Pleb-Dem,GP}.

\subsubsection*{Elliptic type} When $q(z)=1+z^2$, $\d\tau_0+\d\tau_2=0$, and 
$S^2_{0,q}W^*=\{p(z)=p_0z^2+2p_1z+p_2:p_2=-p_0\}$; we may represent $[w]\in
S^2W^*/\spn{q}$ by $-w_2z^2 +2w_1z + w_2$ with $\frac12\{q,w\}=w_1z^2 -2w_2z
-w_1$ and define components of $\xi\in \tor^*$ by $\xi(p)=\xi_1 p_1+\xi_2
p_2$. Modulo constants, the Killing potentials for $\omega_\pm$ are spanned by
\begin{align*}
\mu^+_1 &= -\frac{1 - xy}{1 + xy}, &\mu^+_2 &= - \frac{x + y}{1 + xy}, \\
\mu^-_1 &= -\frac{x + y}{x - y},  &\mu^-_2 &= \frac{1 - xy}{x - y},
\end{align*}
while the barycentric metric $g_0$ and K\"ahler forms $\omega_\pm$ then take
the form:
\begin{align*}
g_0&=\frac{\d x^2}{A(x)} + \frac{\d y^2}{B(y)}
+ \frac{A(x)(\d t_1+(y^2-1)\d t_2)^2}{(x-y)^2(1+xy)^2}
+ \frac{B(y)(\d t_1+(x^2-1)\d t_2)^2}{(x-y)^2(1+xy)^2}\\
\omega_+ &= \frac{\d x \wedge (2 y \, \d t_1 + (y^2 - 1) \d t_2)}{(1 + xy)^2}
+ \frac{\d y \wedge (2 x \, \d t_1 + (x^2 - 1) \d t_2)}{(1 + x y)^2}\\
\omega_- &= \frac{\d x \wedge (2 y \, \d t_1 + (y^2 - 1) \d t_2)}{(x - y)^2}
- \frac{\d y \wedge (2 x \, \d t_1 + (x^2 - 1) \d t_2)}{(x - y)^2}.
\end{align*}
The metrics $g_\pm$ are extremal iff
\begin{equation*}
a_2+b_2=0, \qquad a_0+b_0+a_4+b_4=0,\qquad a_1 + b_1 = a_3 + b_3,
\end{equation*}
in which case
\begin{equation*}
s_+ = 6 (a_3 - b_1) \mu^+_1 - 12 (a_4 + b_0)  \mu^+_2,\qquad
s_- = 12 (a_3 + b_3) \mu^-_1 + 12 (a_4 + b_4)  \mu^-_2.
\end{equation*}
The Bach-flatness condition is therefore:
\begin{equation*}
(a_3 - b_1) (a_3 + b_3)+ 4 (a_4 + b_4) (a_4 + b_0)=0.
\end{equation*}
For $p(z)=1-z^2$, $g=q(x,y)^2g_+/p(x,y)^2$ is CSC iff
$a_2+b_2=0$, $a_0+b_0=0$, $a_4+b_4=0$, and $a_1 + b_1 + a_3 + b_3 = 0$.
For $p(z)=z$, we have instead $a_0+b_0=0$, $a_2+b_2=0$, $a_4+b_4=0$ and
$a_1 - b_1 + a_3 - b_3=0$.

\subsubsection*{Summary table} The following table summarizes the extremal
metric conditions. \smallbreak

{\small
\begin{center}
\begin{tabular}{|c|c|c|c|}
\hline
{\bf Condition}
& {\bf Parabolic type} & {\bf Hyperbolic type} & {\bf Elliptic type}\\[0.5mm]
\hline \hline
$g_\pm$ extremal & 
  $a_0+b_0=0$ & $a_0+b_0=0$ & $a_0 + b_0 + a_4 + b_4 = 0$\\
& $a_1+b_1=0$ & $a_2+b_2=0$ & $a_2 + b_2=0$ \\
& $a_2+b_2=0$ & $a_4+b_4=0$ & $a_1 + b_1 = a_3 + b_3$\\[1mm]
\hline
$g_\pm$ Bach-flat  & extremal and & extremal  and & extremal and \\
& $a_1(a_3 + b_3)=$ & $(a_3 - b_3) (a_1 + b_1)=$ & $(a_3 - b_1)(a_3 + b_3)=$\\ 
& $ - 4 a_0 (a_4 + b_4) $ & $- (a_3 + b_3) (a_1 - b_1)$ &
$- 4 (a_4 + b_4) (a_4 + b_0)$\\[1mm]
\hline
$s_+ \equiv  0$ & extremal and & extremal  and & extremal and \\
($W_+ \equiv 0$)& $a_0 = 0$ & $a_1 = b_1$ & $a_3 =  b_1$\\
                & $a_1 = 0$ & $a_3  = b_3$ & $a_4 =- b_0$\\[1mm]
\hline
$s_- \equiv 0$ & extremal and & extremal  and & extremal and \\
($W_- \equiv 0$)& $a_3 =- b_3$ & $a_1 =- b_1$ & $a_3 =- b_3$\\
                & $a_4 =- b_4$ & $a_3 =- b_3$ & $a_4 =- b_4$\\
\hline
\end{tabular}
\end{center}
}

\medbreak

$g_-$ is never K\"ahler--Einstein, and is a CSC iff $s_-\equiv 0$. The same
holds for $g_+$ except in the parabolic case, when $g_+$ has constant scalar
curvature iff it is extremal with $a_0=0$, and is K\"ahler--Einstein if also
$a_3+b_3=0$.

\appendix

\section{The projective line and transvectants}\label{s:plt}

Let $W$ be a $2$-dimensional real vector space equipped with a symplectic form
$\kappa$ (a non-zero element of $\Wedge^2 W^*$). This defines an isomorphism
$W\to W^*$ sending $u\in W$ to the linear form $u^\flat\colon
v\mapsto\kappa(u,v)$; similarly there is a Lie algebra isomorphism from
$\mathfrak{sl}(W)$ (the trace-free endomorphisms of $W$) to $S^2W^*$ (the
quadratic forms on $W$, under Poisson bracket $\{,\}$) sending $a\in
\mathfrak{sl}(W)$ to the quadratic form $u\mapsto \kappa(a(u),u)$.

The quadratic form $-\det$ on $\mathfrak{sl}(W)$ induces a quadratic form $Q$
on $S^2W^*$ proportional to the discriminant, which polarizes to give an
$\mathfrak{sl}(W)$-invariant inner product $\ipq{p,\tilde p}=Q(p+\tilde
p)-Q(p)-Q(\tilde p)$ of signature $(2,1)$ satisfying the following identity:
\begin{equation}\label{Q-ip}
Q(\{p,\tilde p\}) = \ipq{p,\tilde p}^2 - 4 Q(p) Q(\tilde p).
\end{equation}

The analysis can be made more explicit by introducing a symplectic basis $e_1,
e_2$ of $W$ (so that $\kappa(e_1,e_2)=1$) and hence an affine coordinate $z$
on $\Proj(W)$ (with $[w]=[z([w])e_1+e_2]$). A quadratic form $q\in S^2 W^*$ may
then be written
\begin{equation*}
q(z) = q_0 z^2 + 2 q_1 z + q_2
\end{equation*}
with polarization
\begin{equation*}
q(x,y) = q_0 xy + q_1(x+y)+q_2.
\end{equation*}
In these coordinates the Poisson bracket of $q(z)$ with $w(z)$ is
\begin{gather*}
\{q,w\}(z) = q'(z)w(z)-w'(z)q(z)\qquad\text{with}\\
\{q,w\}_0 = 2 q_0 w_1 - 2 q_1 w_0, \quad
\{q,w\}_1 = q_0 w_2 - q_2 w_0, \quad 
\{q,w\}_2 = 2 q_1 w_2 - 2 q_2 w_1,
\end{gather*}
and the quadratic form and inner product on $S^2W^*$ are
\begin{equation*}
Q(q) = q_1^2 - q_0 q_2\qquad\text{and}\qquad
\ipq{q,p}=  2q_1 p_1 - (q_2 p_0 + q_0 p_2).
\end{equation*}

The elements of $S^m W^*$ may similarly be regarded as polynomials in one
variable of degree at most $m$. For any $n, m\in \N$, the tensor product $S^m
W^* \otimes S^n W^*$ has the following \emph{Clebsch--Gordan} decomposition
into irreducible component:
\begin{equation}
S^m W^* \otimes S^n W^* = \bigoplus_{r=0}^{\min\{m,n\}} S^{m+n - 2r} W^*.
\end{equation}
For any $r = 0, \ldots , \min\{m,n\}$, the corresponding SL$(W)$-equivariant
map $S^m W^* \otimes S^n W^* \to S^{m+n-2r} W^*$ (well-defined up to a
multiplicative constant) is called the \emph{transvectant} of order $r$, and
denoted $(p,q)^{(r)}$---see e.g., Olver~\cite{Olver}. For $m=n$, the
transvectant of order $r$ is symmetric if $r$ is even, and skew if $r$ is
odd. When $p,q$ are regarded as polynomials in one variable, it may be written
explicitly as:
\begin{equation} \label{trans}
(p,q)^{(r)} = \sum_{j=0}^r (-1)^j \binom{n-j}{r-j} \binom{m-r+j}{j} \,
p^{(j)} q ^{(r-j)},
\end{equation}
where $p^{(j)}$ stands for the $j$-th derivative of $p$, with $p ^{(0)} = p$,
and similarly for $q ^{(r-j)}$. In particular, $(p,q)^{(0)}$ is
multiplication, and for any $p,q\in S^2W^*$, $(p,q)^{(1)}$ and $(p,q)^{(2)}$
are constant multiples of the Poisson bracket and inner product respectively.

Elements of $S^mW^*$ (and corresponding polynomials in an affine coordinate)
may be viewed as (algebraic) sections of the degree $m$ line bundle $\cO(m)$
over $\Proj(W)$; in particular, there is a tautological section of
$\cO(1)\otimes W$. The formula~\eqref{trans} for transvectants extends from
algebraic sections to general smooth sections.

\section{Killing tensors and ambitoric conformal metrics}\label{Killing}

The material in this appendix is related to work of
W.~Jelonek~\cite{Jelonek2,Jelonek3a,Jelonek3b} and some well-known results in
general relativity, see \cite{Cosgrove} and \cite{Kamran}.  To provide a
different slant, we take a conformal viewpoint
(cf.~\cite{CD,BGG,Gauduchon,Semm}) and make explicit the connection with
M. Pontecorvo's description~\cite{Pontecorvo} of hermitian structures which
are conformally K\"ahler. We then specialize the analysis to ambitoric
structures.

\subsection{Conformal Killing objects}

Let $(M,c)$ be a conformal manifold. Among the conformally invariant linear
differential operators on $M$, there is a family which are overdetermined of
finite type, sometimes known as twistor or Penrose operators; their kernels
are variously called twistors, tractors, or other names in special
cases. Among the examples where the operator is first order are the equations
for twistor forms (also known as conformal Killing forms) and conformal
Killing tensors, both of which include conformal vector fields as a special
case. There is also a second order equation for Einstein metrics in the
conformal class. Apart from the obvious presence of (conformal) Killing vector
fields and Einstein metrics, conformal Killing $2$-tensors and twistor
$2$-forms are very relevant to the present work.

Let $S^k_0TM$ denote the bundle of symmetric $(0,k)$-tensors $\cS_0$ which are
tracefree with respect to $c$ in the sense that $\sum_i
\cS_0(\eps_i,\eps_i,\cdot)=0$ for any conformal coframe $\eps_i$. In
particular, for $k=2$, $\cS_0\in S^2_0TM$ may be identified with $\sigma_0\in
L^2\otimes\Sym_0(TM)$ via $\alpha\circ\sigma_0(X)=\cS_0(\alpha,c(X,\cdot))$
for any $1$-form $\alpha$ and vector field $X$. Here $\Sym_0(TM)$ is the
bundle of tracefree endomorphisms of $TM$ which are symmetric with respect to
$c$; thus $\sigma_0$ satisfies $c(\sigma_0(X),Y)=c(X,\sigma_0(Y))$ and hence
defines a (weighted) $(2,0)$-tensor $S_0$ in $L^4\otimes S^2_0T^*M$, another
isomorph of $S^2_0TM$ (in the presence of $c$).

A \emph{conformal Killing \textup($2$-\textup)tensor} is a section $\cS_0$ of
$S^2_0TM$ such that the section $\sym_0 D\cS_0$ of $L^{-2}\otimes S^3_0TM$ is
identically zero, where $D$ is any Weyl connection (such as the Levi-Civita
connection of any metric in the conformal class) and $\sym_0$ denotes
orthogonal projection onto $L^{-2}\otimes S^3_0TM$ inside $T^*M\otimes
S^2TM\cong L^{-2} \otimes TM\otimes S^2TM$. Equivalently $\sym D\cS_0 = \sym
(\chi\otimes c)$ for some vector field $\chi$. Taking a trace, we find that
$(n+2)\chi=2\delta^D\cS_0$, where $\delta^D\cS_0$ denotes $\trace_c D\cS_0$,
which may be computed, using a conformal frame $e_i$ with dual coframe
$\eps_i$, as $\sum_i D_{e_i}\cS_0(\eps_i,\cdot)$. Thus $\cS_0$ is conformal
Killing iff
\begin{equation}\label{eq:CK}
\sym D\cS_0 = \tfrac{2}{n+2}\sym (c\otimes\delta^D\cS_0),
\end{equation}
This is independent of the choice of Weyl connection $D$. On the open set
where $\cS_0$ is nondegenerate, there is a unique such $D$ with
$\delta^D\cS_0=0$, and hence a nondegenerate $\cS_0$ is conformal Killing iff
there is a Weyl connection $D$ with $\sym D\cS_0=0$.

A \emph{conformal Killing $2$-form} is a section $\phi$ of
$L^3\otimes\Wedge^2T^*M$ such that $\pi(D\phi)=0$ (for any Weyl connection
$D$) where $\pi$ is the projection orthogonal to $L^3\otimes\Wedge^3T^*M$ and
$L\otimes T^*M$ in $T^*M\otimes L^3\otimes\Wedge^2T^*M$. It is often more
convenient to identify $\phi$ with a section $\Phi$ of $L\otimes
\mathfrak{so}(TM)$ via $\phi(X,Y)=c(\Phi(X),Y)$, where $\mathfrak{so}(TM)$
denotes the bundle of skew-symmetric endomorphisms of $TM$ with respect to
$c$.

\subsection{Conformal Killing tensors and complex structures}\label{A:ckJ}

In four dimensions a conformal Killing $2$-form splits into selfdual and
antiselfdual parts $\Phi_\pm$, which are sections of $L\otimes
\mathfrak{so}_\pm(TM)\cong L^3\otimes\Wedge^2_\pm T^*M$. Following
M. Pontecorvo~\cite{Pontecorvo}, nonvanishing conformal Killing $2$-forms
$\Phi_+$ and $\Phi_-$ describe oppositely oriented K\"ahler metrics in the
conformal class, by writing $\Phi_\pm=\ell_\pm J_\pm$, where $\ell_\pm$ are
sections of $L$ and $J_\pm$ are oppositely oriented complex structures: the
K\"ahler metrics are then $g_\pm=\ell_\pm^{-2}c$. Conversely if
$(g_\pm=\ell_\pm^{-2}c,J_\pm)$ are K\"ahler and $D^\pm$ denote the Levi-Civita
connections of $g_\pm$ then $D^\pm(\ell_\pm J_\pm)=0$ so $\Phi_\pm=\ell_\pm
J_\pm$ are conformal Killing $2$-forms.

The tensor product of sections $\Phi_+$ and $\Phi_-$ of
$L\otimes\mathfrak{so}_+(TM)$ and $L\otimes\mathfrak{so}_-(TM)$ defines a
section $\Phi_+\Phi_-$: as a section of $L^2\otimes\Sym_0(TM)$, this is simply
the composite ($\Phi_+\circ\Phi_-=\Phi_-\circ\Phi_+$); as a section of
$L^4\otimes S^2_0T^*M$ it satisfies
$(\Phi_+\Phi_-)(X,Y)=c(\Phi_+(X),\Phi_-(Y))$.

When $\Phi_\pm=\ell_\pm J_\pm$ are nonvanishing, $\Phi_+\Phi_-=\ell_+\ell_-
J_+J_-$ is a symmetric endomorphism with two rank $2$ eigenspaces at each
point.  Conversely if $\sigma_0$ is such a symmetric endomorphism, we may
write $\sigma_0=\ell^2J_+J_-$ for uniquely determined almost complex
structures $J_\pm$ up to overall sign, and a positive section $\ell$ of $L$.

\begin{prop} A nonvanishing section $\sigma_0=\ell^2 J_+J_-$ of
$L^2\otimes \Sym_0(TM)$ \textup(as above\textup) is associated to a conformal
Killing $2$-tensor $\cS_0$ iff $J_\pm$ are integrable complex structures which
are ``K\"ahler on average'' with length scale $\ell$, in the sense that if
$D^\pm$ denote the canonical Weyl connections of $J_\pm$, then the connection
$D=\frac 12 (D^++D^-)$ preserves the length scale $\ell$ \textup(i.e.,
$D^+\ell+D^-\ell=0$\textup).

If these equivalent conditions hold, then also $\sym D\cS_0=0$.
\end{prop}

\noindent (With respect to an arbitrary metric $g$ in the conformal class, the
``K\"ahler on average'' condition means that the Lee forms $\theta_\pm^g$
satisfy $d(\theta_g^++\theta_g^-)=0$. In the case that $J_+$ and $J_-$ both
define conformally K\"ahler metrics $g_\pm$, the metric $g_0=\ell^{-2}c$ is
the barycentric metric with $g_0=f\,g_+=f^{-1}g_-$ for some function $p$.)

\begin{proof} Let $D$, $D^+$, $D^-$ be Weyl connections with
$D=\frac 12 (D^++D^-)$ in the affine space of Weyl connections. (Thus the
induced connections on $L$ are related by $D=D^++\theta=D^--\theta$ for some
$1$-form $\theta$.) Straightforward calculation shows that
\begin{equation*}
D\sigma_0 = D(\ell^2)\otimes J_+\circ J_- +\ell^2
\bigl(D^+J_+\circ J_- + J_+\circ D^- J_-\bigr)+R
\end{equation*}
where $R$ is an expression (involving $\theta$) whose symmetrization vanishes
(once converted into a trilinear form using $c$). If $J_\pm$ are integrable
and K\"ahler on average, then taking $D^\pm$ to be the canonical Weyl
connections and $\ell$ the preferred length scale, $\ell^2J_+J_-$ is thus
associated to a conformal Killing tensor $\cS_0$ with $\sym D\cS_0=0$.

For the converse, it is convenient (for familiarity of computation) to work
with the associated $(2,0)$-tensor $S_0$ with $S_0(X,Y)=\ell^2c(J_+J_-X,Y)$.
Since $S_0$ is nondegenerate, and associated to a conformal Killing tensor, we
can let $D=D^+=D^-$ be the unique Weyl connection with $\sym DS_0=0$: note
that $\sym\colon L^4 \otimes T^*M\otimes S^2T^*M\to L^4\otimes S^3T^*M$ here
becomes the natural symmetrization map. Thus
\begin{equation*}
\sum_{X,Y,Z} D_X(\ell^2) c(J_+\circ J_-Y,Z)=
\sum_{X,Y,Z} \ell^2\Bigl(c\bigl((D_XJ_+) J_-Y,Z\bigr)
+c\bigl(J_+(D_XJ_-)Y,Z\bigr)\Bigr),
\end{equation*}
where the sum is over cyclic permutations of the arguments. If $X,Y,Z$ belong
to a common eigenspace of $S_0$ then the right hand side is zero---this
follows because, for instance, $c\bigl((D_XJ_\pm)J_\pm Y,Z\bigr)$ is skew in
$Y,Z$ whereas the cyclic sum of the two terms is totally symmetric.

It follows that $D\ell=0$, hence the right hand side is identically zero in
$X,Y,Z$. Additionally $c(D_XJ_\pm \cdot ,\cdot)$ is
$J_\pm$-anti-invariant. Thus these $2$-forms vanish when their arguments have
opposite types ($(1,0)$ and $(0,1)$) with respect to the corresponding complex
structure. Now suppose for example that $Z_1$ and $Z_2$ have type $(1,0)$ with
respect to $J_+$, but opposite types with respect to $J_-$ ($J_+$ and $J_-$
are simultaneously diagonalizable on $TM\otimes \C$). Then by substituting
first $X=Y=Z_1$, $Z=Z_2$ into
\begin{equation*}
\sum_{X,Y,Z} c\bigl((D_XJ_+) J_-Y,Z\bigr)
=\sum_{X,Y,Z} c\bigl((D_XJ_-)Y,J_+Z\bigr),
\end{equation*}
and then $X=Y=Z_2$, $Z=Z_1$, we readily obtain
\begin{equation*}
c\bigl((D_{Z_1}J_+) Z_1,Z_2\bigr)=0=c\bigl((D_{Z_2}J_+) Z_1,Z_2\bigr).
\end{equation*}
Thus $D_{J_+X}J_+=J_+D_XJ_+$ for all $X$ and $J_+$ is integrable.
Similarly, we conclude $J_-$ is integrable.
\end{proof}

Since $D$ is the Levi-Civita connection $D^g$ of the ``barycentric'' metric
$g=\ell^{-2}c$, it follows that $S_0=g(J_+J_-\cdot,\cdot)$ is a \emph{Killing
  tensor} with respect to $g$, i.e., satisfies $\sym D^gS_0=0$ iff $J_+$ and
$J_-$ are integrable and K\"ahler on average, with barycentric metric
$g$. More generally, we can use this result to characterize, for any metric
$g$ in the conformal class and any functions $f,h$, the case that
\begin{equation} \label{S}
S (\cdot, \cdot) = f \, g (\cdot, \cdot)+  h \, g (J_+J_- \cdot, \cdot),
\end{equation}
is a Killing tensor with respect to $g$. If $\theta_\pm$ are the Lee forms of
$(g,J^\pm)$, i.e., $D^\pm=D^g\pm \theta_{\pm}$, then we obtain the following
more general corollary.

\begin{cor} \label{corvesti} $S=f \, g 
+h \, g (J_+J_- \cdot, \cdot)$, with $h$ nonvanishing, is a Killing tensor
with respect to $g$ if and only if:
\begin{gather}\label{v1}
\text{$J_+$ and $J_-$ are both integrable};\\
\label{v2} \theta_+ + \theta_- = -\frac{\d h}{h};\\
\label{v3} J_+ \d f = J_- \, \d h. 
\end{gather}
\end{cor}
\noindent (Obviously when $h$ is identically zero, $S$ is a Killing tensor iff
$f$ is constant.)

\subsection{Conformal Killing tensors and the Ricci tensor}

The tracefree part $\rc{g}_0 = \rc{g}-\frac 1n \s{g} g$ of the Ricci tensor
of a compatible metric $g=\mu_g^{-2}c$ on a conformal $n$-manifold $(M,c)$
defines a tracefree symmetric $(0,2)$-tensor
$\cS_0^g(\alpha,\beta)=\rc{g}_0(\alpha^\sharp,\beta^\sharp)$ (where for
$\alpha\in T^*M$, $g(\alpha^\sharp,\cdot)=\alpha)$), where the corresponding
section of $L^4\otimes S^2_0T^*M$ is $\mu_g^4\rc{g}_0$.

The differential Bianchi identity implies that $0=\delta^g (\rc{g}-\frac 12
\s{g} g)= \delta^g\rc{g}_0-\frac{n-2}{2n} \d\s{g}$. Hence the following are
equivalent:
\begin{bulletlist}
\item $\cS_0^g$ is a conformal Killing tensor;
\item $\sym D^g\cS_0^g = \frac{n-2}{n(n+2)} \sym (g^{-1}\otimes \d\s{g})$;
\item $\rc{g}-\frac 2{n+2} \s{g} g$ is a Killing tensor with respect to $g$;
\item $D^g_X\rc{g}(X,X) = \frac 2{n+2} \d\s{g}(X) g(X,X)$ for all vector
fields $X$.
\end{bulletlist}
Riemannian manifolds $(M,g)$ satisfying these conditions were introduced by
A. Gray as $\mathcal A C^\perp$-manifolds~\cite{Gray}. Relevant for this paper
is the case $n=4$ and the assumption that $\rc{g}$ has two rank $2$
eigendistributions, which has been extensively studied by
W.~Jelonek~\cite{Jelonek3a,Jelonek3b}.

Supposing that $g$ is not Einstein, Corollary~\ref{corvesti} implies, as shown
by Jelonek, that
\begin{equation*}
\rc{g}-\tfrac 13 \s{g} g = f \, g +  h \, g (J_+J_- \cdot, \cdot)
\end{equation*}
is Killing with respect to $g$ iff~\eqref{v1}--\eqref{v3} are satisfied. Since
$J_\pm$ are both integrable, Jelonek refers to such manifolds as
\emph{bihermitian Gray surfaces}. It follows from~\cite{AG1} that both
$(g,J_+)$ and $(g,J_-)$ are conformally K\"ahler, so that in the context of
the present paper, a better terminology would be \emph{ambik\"ahler Gray
  surfaces}.

However, the key feature of such metrics is that the Ricci tensor is
$J_\pm$-invariant: as long as $J_\pm$ are conformally K\"ahler,
Proposition~\ref{p:diagonal-ambi} applies to show that the manifold is either
ambitoric or of Calabi type; it is not necessary that the $J_\pm$-invariant
Killing tensor constructed in the proof is equal to the Ricci tensor $\rc{g}$.

Jelonek focuses on the case that the ambihermitian structure has Calabi
type. This is justified by the global arguments he employs. In the ambitoric
case, there are strong constraints, even locally.

\subsection{Killing tensors and hamiltonian $2$-forms}

The notion of hamiltonian $2$-forms on a K\"ahler manifold $(M,g,J,\omega)$
has been introduced and extensively studied in \cite{ACG,ACG2}. According to
\cite{ACG2}, a $J$-invariant $2$-form $\phi$ is hamiltonian if it satisfies
\begin{equation}\label{hamiltonian-new}
D_X \phi =
\frac{1}{2} \Big( \d\sigma  \wedge J X^\flat - J \d\sigma \wedge X^\flat \Big),
\end{equation}
for any vector field $X$, where $X^\flat= g(X)$ and $\sigma={\rm tr}_{\omega}
\phi = g(\phi, \omega)$ is the trace of $\phi$ with respect to $\omega$. An
essentially equivalent (but not precisely the same) definition was given in
the four dimensional case in \cite{ACG}, by requiring that a $J$-invariant
$2$-form $\varphi$ is closed and its primitive part $\varphi_0$ satisfies
\begin{equation}\label{hamiltonian-old}
D_X \varphi_0 = -\frac{1}{2} \d\sigma(X)\omega
+ \frac12 \Big( \d\sigma\wedge J X^\flat - J \d\sigma \wedge X^\flat \Big),
\end{equation}
for some smooth function $\sigma$.  Note that, in order to be closed,
$\varphi$ is necessarily of the form $\frac{3}{2} \sigma \omega + \varphi_0$.

The relation between the two definitions is straightforward:
$\varphi=\frac{3}{2} \sigma \omega + \varphi_0$ is closed and verifies
\eqref{hamiltonian-old} iff $\phi = \varphi_0 + \frac{1}{2}\sigma \omega$
satisfies \eqref{hamiltonian-new}.

Specializing Corollary~\ref{corvesti} to the case when the metric $g$ is
K\"ahler with respect to $J=J_+$ allows us to identify $J$-invariant symmetric
Killing tensors with hamiltonian $2$-forms as follows:

\begin{prop}\label{kahler-case} Let $S$ be a symmetric $J$-invariant tensor
on a K\"ahler surface $(M,g,J,\omega)$, and $\psi(\cdot, \cdot)= S(J\cdot,
\cdot)$ be the associated $J$-invariant $2$-form. Then $S$ is Killing iff
$\phi = \psi - ({\rm tr}_{\omega}\psi) \omega$ is a hamiltonian $2$-form
\textup(i.e., verifies \eqref{hamiltonian-new}\textup).
\end{prop}
\begin{proof}
As observed in \cite[p.~407]{ACG2}, $\phi$ satisfies \eqref{hamiltonian-new}
iff $\varphi= \phi + ({\rm tr}_{\omega} \phi) \omega$ is a closed $2$-form and
$\psi= \phi -({\rm tr}_{\omega} \phi) \omega$ is the $2$-form associated to a
$J$-invariant Killing tensor (this is true in any complex dimension $m>1$).

Noting that the $2$-forms $\varphi$ and $\psi$ are related by $\varphi = \psi
- \frac{2 {\rm tr}_{\omega} \psi}{m-1}\omega$ , we claim that in complex
dimension $m=2$, the $2$-form $\varphi= \phi -2 ({\rm tr}_{\omega} \psi)
\omega$ is automatically closed, provided that $\psi$ is the $2$-form
associated to a $J$-invariant Killing tensor $S$. Indeed, under the K\"ahler
assumption the conditions~\eqref{v1}--\eqref{v2} specialize as
\begin{equation} \label{v1K} \text{$J_-$ is integrable},
\end{equation}
\begin{equation} \label{v2K} \theta_- = - \, \frac{\d h}{h}, 
\end{equation}
It follows that $(g_- = h^{-2} \, g, J_-, \omega_- = g_- (J_- \cdot, \cdot))$
defines a K\"ahler metric.  From \eqref{S} we have
\begin{equation} \psi = f \, \omega_+ +  h^3 \, \omega_-,
\end{equation}
where $\omega_+ = g (J_+ \cdot, \cdot)$ denotes the K\"ahler form
of  $(g, J_+)$. In particular, the trace of $\varphi$
with respect to $\omega_+$ is equal to $2 f$ while the  condition
(\ref{v3}) and the fact that $\omega_-$ is closed imply that $\varphi = 
\psi - 4f \,
\omega_+ = - 3 f \, \omega_+ + h^3 \, \omega_-$ is
closed too.
\end{proof}

\subsection{Killing tensors associated to ambitoric structures}

We have seen in the previous subsections that there is a link between Killing
tensors and ambihermitian structures. We now make this link more explicit in
the case of ambitoric metrics.

In the ambitoric situation, the barycentric metric $g_0$ (see
section~\ref{s:loc-class}) satisfies $\theta^0_+ + \theta^0_- = 0$. It then
follows from Corollary~\ref{corvesti} that the (tracefree) symmetric bilinear
form $g_0 (I \cdot, \cdot)$ (with $I = J_+ \circ J_-$) is Killing with respect
to $g_0$. More generally, let $g$ be any $(K_1, K_2)$-invariant riemannian
metric in the ambitoric conformal class $c$, so that $g$ can be written as $g
= h\, g_0$ for some positive function $h (x, y)$, where $x, y$ are the
coordinates introduced in section~\ref{s:loc-class}. Then $\theta^g_+ +
\theta^g_- = - \d\log{h}$. From Corollary~\ref{corvesti} again, the symmetric
bilinear form $S_0(\cdot, \cdot)=h \, g(I \cdot, \cdot)$ is conformal Killing.
Moreover, by condition (\ref{v3}) in Proposition~\ref{corvesti}, it can be
completed into a Killing symmetric bilinear form $S= f \, g + S_0$ iff the
$1$-form $\d h \circ I$ is closed. Since $I \d x = - \d x$ and $I \d y = \d y$,
$\d h \circ I$ is closed iff $h_x \, \d x - h_y \, \d y$ is closed, iff $h_{x
y} = 0$; the general solution is $h (x, y) = F (x) - G (y)$, for some
functions $F, G$. Note that the coefficient $f(x,y)$ is determined by $\d f
=-I \d h = F'(x)\d x + G'(y)\d y$ (see \eqref{v3}), so we can take without loss
$f(x,y)= F(x) + G(y)$.  Thus, $S$ is Killing, with eigenvalues (with respect
to $g$) equal to $2F (x)$ and $2G (y)$.

A similar argument shows that any metric of the form $g= f(z) g_0$, where
$g_0$ is the barycentric metric of an ambik\"ahler pair of Calabi type and $z$
is the momentum coordinate introduced in section~\ref{s:calabi-type}, admits a
nontrivial symmetric Killing tensor of the form $S(\cdot, \cdot)= f(z)g(\cdot,
\cdot) + f(z) g(I\cdot, \cdot)$ (and hence with eigenvalues $(2f(z), 0)$).

It follows that there are infinitely many $\tor$-invariant metrics in a given
ambitoric conformal class, which admit nontrivial symmetric Killing tensors.

There are considerably fewer such metrics with diagonal Ricci tensor.  By
Proposition~\ref{p:diagonal-ricci} these have the form $g=h(x,y) g_0$ where
$h(x,y)=(x-y)q(x,y)/p(x,y)^2$. In order for $g$ to admit a nontrivial
symmetric Killing tensor, we must have $h_{xy}=0$. A calculation shows that
this happens iff $Q(p)=0$ (i.e., $p(z)$ has repeated roots). Since $p$ is
orthogonal to $q$, this can only happen if $Q(q)\geq 0$ and there are
generically ($Q(q)>0$) just two solutions for $p$, which coincide if $Q(q)=0$.

\end{document}